\documentclass[11pt]{article}
\usepackage[a4paper,top=3cm,bottom=3.5cm,left=2.5cm,right=2.5cm,marginparwidth=1.75cm]{geometry}
\usepackage[T1]{fontenc}
\usepackage{graphicx} 
\usepackage{amsmath,amssymb,amsthm}
\usepackage{matlab-prettifier}

\usepackage{algpseudocode}
\usepackage[title]{appendix}
\usepackage{subcaption}
\newtheorem{theorem}{Theorem}
\newtheorem{lemma}{Lemma}
\newtheorem{remark}{Remark}
\usepackage[citestyle=numeric-comp,]{biblatex}
\usepackage{hyperref}
\hypersetup{colorlinks,allcolors=black}

\usepackage{matlab-prettifier}
\usepackage{caption}
\usepackage{subcaption}
\usepackage{appendix}
\usepackage{algorithm}
\usepackage{algpseudocode}
\newcommand{\floor}[1]{\lfloor #1 \rfloor}
\usepackage{siunitx}
\usepackage{booktabs}
\usepackage{amsmath}  

\newcommand{\norm}[1]{\left\lVert#1\right\rVert}
\newcommand{\dif}{\mathrm{d}}
\newcommand{\abs}[1]{\left|{#1}\right|}
\newcommand{\ii}{\mathrm{i}}
\newcommand{\x}{\textbf{x}}

\newcommand{\R}{\mathbb{R}}

\title{Fast Singular-Kernel Convolution on General Non-Smooth Domains via Truncated Fourier Filtering}
\author{Oscar Bruno \and Jinghao Cao\thanks{\footnotesize Computing and Mathematical Sciences Department, California Institute of Technology, Pasadena, CA 91125, USA (obruno@caltech.edu, jinghao.cao@caltech.edu)} }
\date{}
\begin{document}
\maketitle
\begin{abstract}
The rapid and accurate evaluation of convolutions with singular kernels plays crucial roles in a wide range of scientific and engineering applications. These include convolutions with $1/r^{\alpha}$ kernels arising in fractional diffusion, as well as $\log(r)$ and $1/r$-type singular kernels commonly encountered in potential theory, acoustics, electromagnetic scattering, and quantum mechanics. Building on the recently introduced Truncated Fourier Filtering method (TFF) for smooth kernels, and focusing on logarithmic singularities for definiteness, this work presents a fast, high-order numerical methodology that extends the approach to singular kernels and non-smooth domains. The method relies on truncated Fourier expansions of a prescribed order $F$ for the characteristic function of the integration domain as well as expansions of the products of certain 1D characteristic functions and singular functions. Combined with an $F$-dependent number of trapezoidal rule integration nodes, the algorithm achieves high-order accuracy in spite of the severe approximation errors typically associated with truncated Fourier expansions of discontinuous and singular functions. A comprehensive theoretical analysis is provided that explains the surprising and advantageous properties of the proposed approach. A variety of numerical examples demonstrate the method’s performance and effectiveness. 
\end{abstract}
\section{Introduction}
The rapid and accurate evaluation of convolutions integrals
\begin{equation}
\label{eq:integral}
    \int_{\Omega} K (x-y) \phi(y) \dif y,
\end{equation}
with singular kernels is a 
fundamental challenge that arises across a broad spectrum of scientific and 
engineering applications. Problems involving $1/r^{\alpha}$ kernels in fractional 
diffusion, as well as logarithmic and $1/r$-type singular kernels in potential 
theory, acoustics, electromagnetic scattering, and quantum mechanics, all require 
numerical treatments that remain efficient and reliable in the presence of 
singular behavior and geometric complexity. Motivated by these demands, we build upon the recently introduced
smooth-kernel Truncated Fourier Filtering (TFF) methodology
of~\cite{br_pooj_2025}, whose order of accuracy deteriorates in the
presence of non-smooth kernels, and extend it to handle singular kernels
and non-smooth domains with high-order accuracy. For smooth kernel
functions, the TFF approach evaluates the associated convolution
integrals on the basis of an equispaced mesh over the convolution
domain—regardless of the domain’s geometric complexity, including cases
involving corners or even cusps. The method relies on truncated Fourier
expansions of prescribed order $F$ for the characteristic function of
the integration domain, as well as expansions of products of certain
one-dimensional characteristic functions with singular factors. The
present contribution extends the TFF framework to enable high-order
evaluation of convolution integrals with singular kernels. In addition
to introducing the method, we analyze its key properties and demonstrate
its performance through representative numerical examples.

As suggested above, the evaluation of integrals of the form
\eqref{eq:integral} with singular kernels plays a central role in
numerical analysis, applied mathematics, and numerous applications.
Recent developments have been driven by challenges arising in boundary
integral equation formulations, high-contrast media, and time-domain
scattering problems, where singular kernels occur naturally and
computational efficiency is essential. A wide variety of methodologies
have been proposed for the fast and accurate evaluation of
convolution-type integrals with singular kernels.

Early work on endpoint-corrected trapezoidal rules and high-order
quadratures~\cite{rokhlin1990end,alpert1995high,kapur1997high,alpert1999hybrid}
achieved high-order accuracy for algebraic and logarithmic
singularities, though in the comparatively simpler setting of
one-dimensional integration. The method of~\cite{BrunoKunyansky2001}
introduced polar changes of variables centered at the singularity to
isolate the radial singular behavior; combined with a global partition
of unity, this approach enabled highly accurate quadratures even in the
presence of geometric irregularities on general smooth and non-smooth
two-dimensional surfaces in three-dimensional space. Subsequent
developments, such as the Chebyshev-based rectangular–polar
method~\cite{bruno2020chebyshev}, removed the need for an overall
partition of unity and demonstrated fast, high-order performance.

Other approaches include the density-interpolation
method~\cite{perez2019planewave}, which mollifies the kernel singularity
by subtracting an interpolant of the density, thereby enabling
high-order convergence, and the Quadrature by Expansion (QBX)
technique~\cite{BarnettGreengard2011}, which constructs local analytic
expansions of layer potentials to obtain spectrally accurate evaluations
near or on boundaries for smooth geometries. Fast convolution
methods~\cite{vico2016fast} employ smooth approximations of the
underlying Green function, achieving high-order accuracy for singular
convolution integrals over the full space $\mathbb{R}^2$, provided the
convolution density is defined and smooth throughout $\mathbb{R}^2$.

Additional contributions include the highly effective algorithms of
\cite{helsing2009integral,helsing2022solving,helsing2013solving} for
one-dimensional boundary Fredholm second-kind equations on non-smooth
curves. The method of~\cite{helsing2022solving}, in particular, relies
on the RCIP framework, which involves nontrivial preconditioner
construction, recursive refinement, and specialized discretizations. This
machinery is powerful, but it renders the approach difficult to adopt
and less straightforward to generalize to two-dimensional singular
integrals. Finally, FMM-accelerated Poisson solvers based on pixel-wise
smooth extensions~\cite{fryklund2023fmm} have been demonstrated for
slowly oscillatory convolution densities.

To summarize, despite their strengths, existing methods typically
require one of the following: (i) carefully engineered local coordinate
transformations, (ii) analytic expansions centered at boundary points,
or (iii) densities defined throughout all of space. In contrast, the
methodology developed in this work preserves the core philosophy of the
TFF approach: it avoids geometric reparameterization, operates entirely
on equispaced grids, and exploits truncated Fourier representations,
even for discontinuous or singular factors, to achieve high-order
accuracy. This results in a conceptually simple, grid-based algorithm
that is robust for arbitrary domains, including those with corners or
limited smoothness, and attains high-order accuracy for a broad class of
singular kernels. Similar to~\cite{BrunoKunyansky2001}, our approach
also employs local integration in polar coordinates; however, the
underlying mechanism is fundamentally different. Rather discretely evaluating the integral in a polar coordinate system,
the TFF methodology leverages truncated Fourier expansions of
characteristic functions and singular factors.

This paper is organized as follows. Section~\ref{sec:preliminaries}
presents the necessary background concerning the smooth formulation
of the TFF method developed in~\cite{br_pooj_2025}. Section~\ref{sec:decomposition}
then extends these results and establishes the convergence rate of the
truncated Fourier filtering method for general singular kernels
(Theorem~\ref{mainlemma}). This section also introduces a windowed
decomposition of the convolution integral into two components and
describes algorithmic strategies for evaluating each contribution.
Within the support of the window function, the logarithmic singularity
is replaced by its Fourier series, for which the TFF method guarantees rapid convergence; the associated angular integral is handled through a
somewhat delicate application of the same technique. Outside the window
support, the integrand is smooth, and an additional use of the TFF
strategy, combined with the Fast Fourier Transform, yields a fast,
high-order evaluation of the remaining contribution. Finally,
Section~\ref{sec:numerics} summarizes the complete algorithm and
presents numerical results illustrating the accuracy and efficiency of
the proposed methodology.

\section{Preliminaries}
\label{sec:preliminaries}
Throughout this work we denote by $Q_{N,[a,b]}$ the discrete trapezoidal-rule quadrature operator over the interval $[a,b]$ based of the $N+1$ quadrature points $x_k = a + \frac{k}{N}(b-a), k=0,\dots, N$---so that,  for, say, a piecewise continuous function $\varphi$ defined in the interval $[a,b]$ we have
\begin{equation}
    \label{eq:trapz}
    \int_{a}^b \varphi(b)\ \dif x  \approx Q_{N,[a,b]}\left[ \varphi \right] := \frac{b-a}{N}\sum_{k=0}^{N-1} \frac{\varphi(x_k)+\varphi(x_{k+1})}{2}.
\end{equation}
If $\Phi$ is a function of two real variables, we denote the 2D trapezoidal rule on the rectangle $[a_1,b_1]\times [a_2,b_2] \subset \mathbb{R}^2$ by 
\begin{equation}
    \label{eq:2dtrapz}
        Q_{N_2,[a_2,b_2]}^2[Q_{N_1,[a_1,b_1]}^1[\Phi]],
\end{equation}
where $Q_{N_i,\in[a_i,b_i]}^i(i=1,2)$ denotes the one-dimensional trapezoidal rule in the $i$-th variable.

In one-dimension, the following Lemma expresses the error in the trapezoidal rule approximation of the integral in~\eqref{eq:trapz}, with $a = -\pi$ and $b=\pi$, in terms of the Fourier coefficients $\varphi$ in the interval $[-\pi,\pi]$.
\begin{lemma}
\label{lem:trapz}
    Let $\varphi$ denote a piecewise continuous function on $[-\pi,\pi]$, with Fourier series $\varphi(x) = \sum_{k\in\mathbb{Z}}\varphi_{k}e^{\ii k\pi}$. Then, the trapezoidal-rule quadrature error is given by
    \begin{equation}
        Q_{N,[-\pi,\pi]}[\varphi] - \int_{-\pi}^\pi \varphi(x)\ \dif x = 2\pi\sum_{i\in \mathbb{N}\setminus \{0\}}\phi_{iN}.
    \end{equation}
\end{lemma}
\begin{proof}
  In view of the convergence of the Fourier series of $\varphi$ in
  $L^2[-\pi,\pi]$, the proof follows easily from the exact expression
  for the action of $Q_{N,[-\pi,\pi]}$ on the Fourier basis function
  $E_k(x):=e^{\ii k x}$:
    \begin{equation}
         Q_{N,[-\pi,\pi]} (E_k) 
        =\frac{2\pi}{N} \sum_{j=0}^{N-1} e^{\frac{2\pi \ii k}{N}j} =\begin{cases}
            2\pi & \text{if $k$ is a multiple of $N$},\\
            0 &\text{otherwise}.
        \end{cases}
    \end{equation}
\end{proof}

\begin{theorem}\cite{br_pooj_2025}
\label{thm:krishna}
For $P\in \mathbb{N}$, let $K$, $K_j$ denote compact sets with a Lipschitz boundary such that $ K=\bigcup_{j=1}^{P}K_j$ and $D\supset K$ a rectangle in $\R^m$. Further, let   $f\in C_\mathrm{pw}^{\infty}(K)$ be such that the restriction $f_j = f|_{K_j}$ of $f$ to the set $K_j$ satisfies $f_j\in C^{\infty}(K_j)$, and let $\widetilde{f}_j$ be a smooth and $D$-periodic extension of $f_j$. Let $F\in\mathbb{N}^{m}$, $n\in\mathbb{N}^{m}$ such that
$n_{r} \geq 2F_r$ for each $1\leq r\leq m$. Then, for any positive integer $p$ there is a constant $C$ such that the total error satisfies
$$\left|\int_{D}\widetilde{f}(\x) d\x -
T_{n}^{m,D}\left[\chi_{K}^{F}\widetilde{f}\right]\right| \leq C\frac{\log(F_{\mathrm{max}})}{F_{\mathrm{min}}^{p}},$$
where $ F_{\mathrm{min}} = \min_{1\leq r\leq m}\{F_r\}$ and $F_{\mathrm{max}} = \max_{1\leq r\leq m}\{F_r\}$.
\end{theorem}
The present paper utilizes a generalized version of Theorem~1, which extends the result to general singular functions $g$ with an integrable singularity instead of the characteristic function $\chi_K^F$ considered in that theorem. We state and prove the more general result in what follows in the 1D case---the only version of this theorem needed in this paper. A proof in higher dimensions may obtained along the lines of the proof of~\cite[Theorem 1]{br_pooj_2025}.

\section{Window decomposition of the integral}
\label{sec:decomposition}
Throughout the paper the results and ideas are illustrated by consideration of the 2D domain $\Omega$ bounded by the piecewise-smooth but non-smooth curve
\begin{equation}
    {\partial \Omega} = \{(l_x\sin(t/2),-l_y\sin(t))\ |\ t\in [0,2\pi]\}
\end{equation}
with $l_x,l_y>0$; as shown in what follows, however,  the proposed method  and ideas can be applied to any  domain with a piecewise-smooth boundary. 
\begin{figure}[ht]
    \centering
    \includegraphics[width=0.5\linewidth]{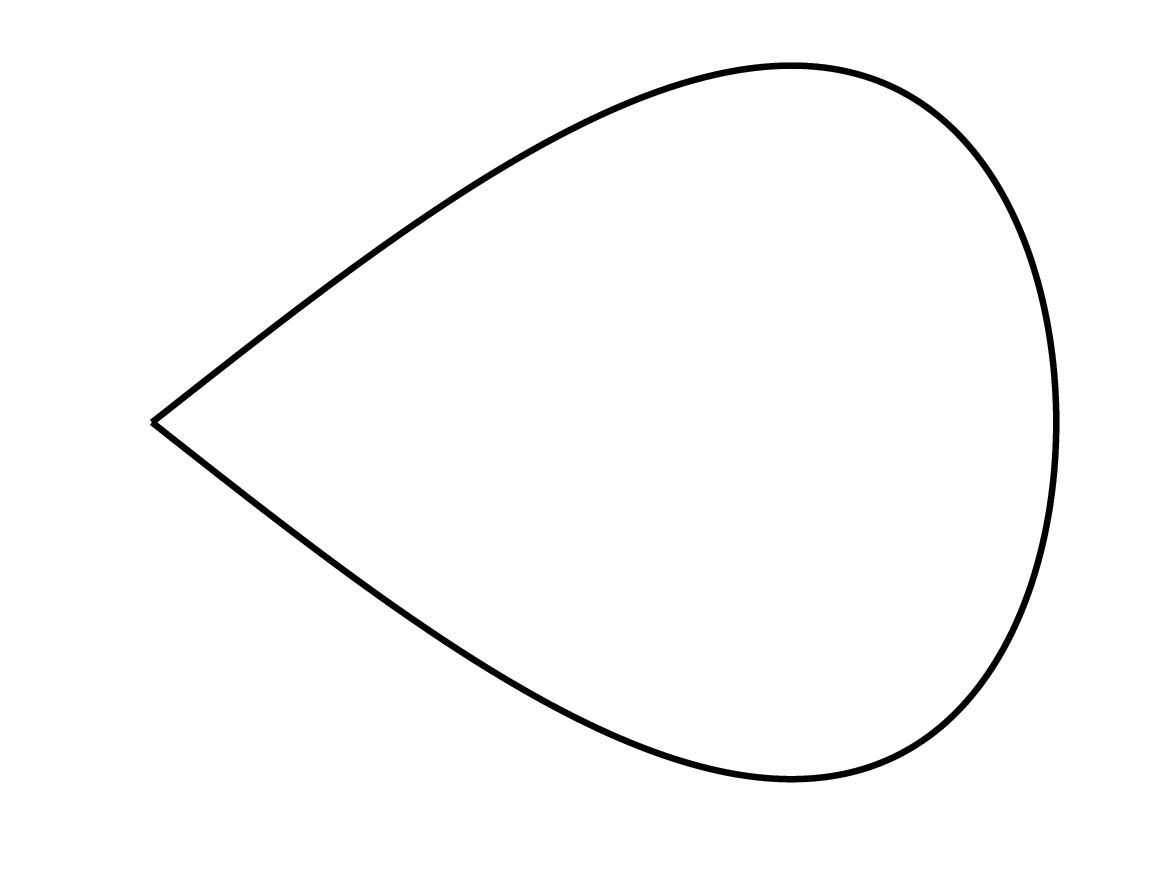}
    \caption{Illustration of a drop shaped domain utilized in the description of the proposed TFF algorithm, with parameterization given by $ {\partial\Omega} = \{(3\sin(t/2),-2\sin(t))\ |\ t\in [0,2\pi]\}$. }
\label{fig:teardrop}
\end{figure}

In what follows we focus on the evaluation of the convolution integral
\begin{equation}
   \int_\Omega \log \abs{x-y} \phi(y)\ \dif y
\end{equation}
To isolate for the logarithmic singularity we utilize the radial window function which, 
for $x\in\Omega$, is given by
\begin{equation}
\label{eq:windowfunction}
    W_1(r) = 
    \begin{cases}
        1&\ 0 \leq r  < w_0, \\
        \mathrm{exp}(\frac{2 \mathrm{exp}(-1/u)}{u-1})&\ w_0 \leq r  \leq w_1,  \\
        0&\ w_1 < r,
    \end{cases}
    \text{with } u = \frac{\abs{r}-w_0}{w_1-w_0},
\end{equation}
where  $w_0$ and $w_1$ are the ``window widths'' ($0< w_0 < w_1$.) Using the window function $W_1$ we decompose the integral as the sum of two terms
\begin{equation}
\label{eq:splitmainintegral}
\begin{split}
    &\int_\Omega \log \abs{x-y} \phi(y)\ \dif y\\
    =& \int_\Omega \log \abs{x-y} \phi(y) W_1(\abs{x-y})\ \dif y + \int_\Omega \log \abs{x-y} \phi(y) (1-W_1(\abs{x-y}))\ \dif y.
\end{split}
\end{equation}
Calling
\begin{equation}\label{I1I2}
  I_1(x)  = \int_\Omega \log \abs{x-y} \phi(y) W_1(\abs{x-y})\ \dif y \quad\mbox{and}\quad I_2(x)  = \int_\Omega \log \abs{x-y} \phi(y) (1-W_1(\abs{x-y}))\ \dif y,
\end{equation}
in the next two sections we discuss the computational methods used to obtain these two contributions to the convolution integral.
\begin{figure}[ht]
\begin{subfigure}[b]{0.5\textwidth}
\includegraphics[width=\linewidth]{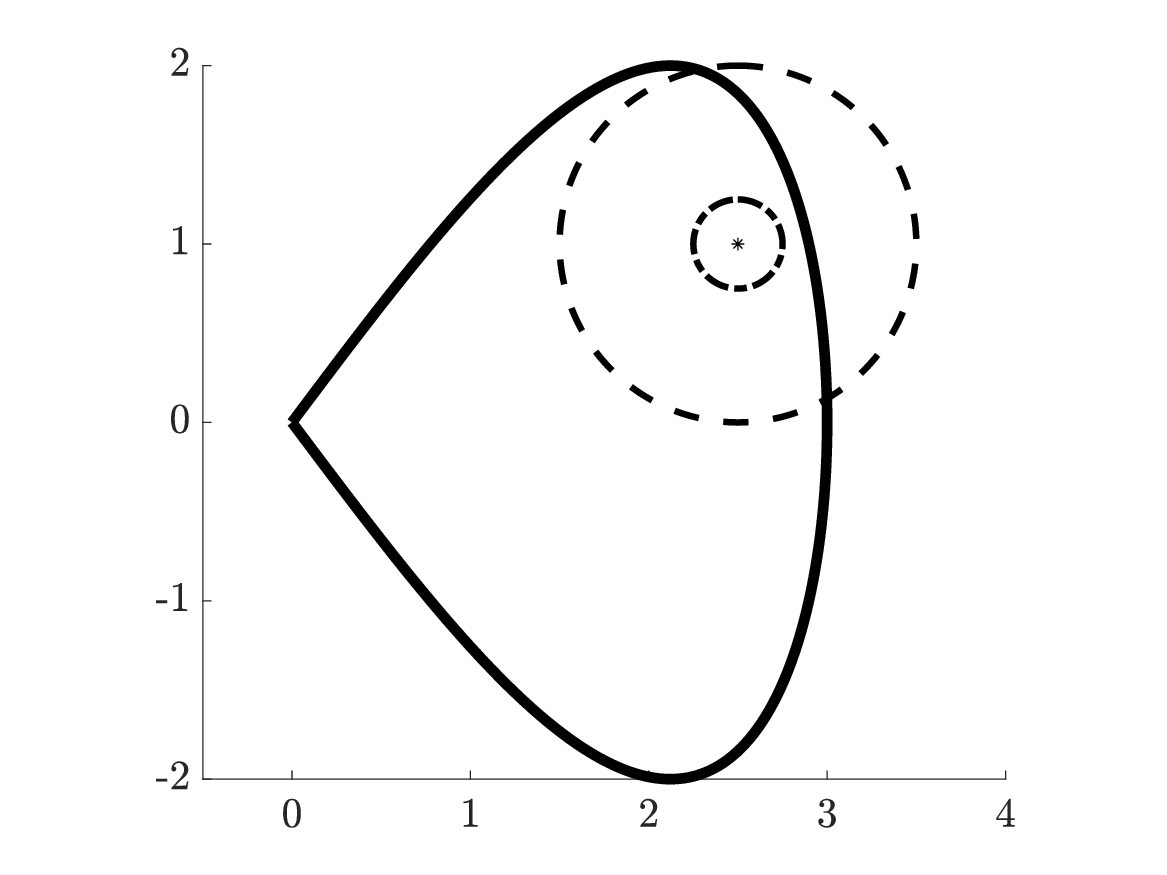}
\caption{Visualization of $W_1$ at $x=(2.5, 1)$.}
\end{subfigure}
\begin{subfigure}[b]{0.48\textwidth}
    \includegraphics[width=\linewidth]{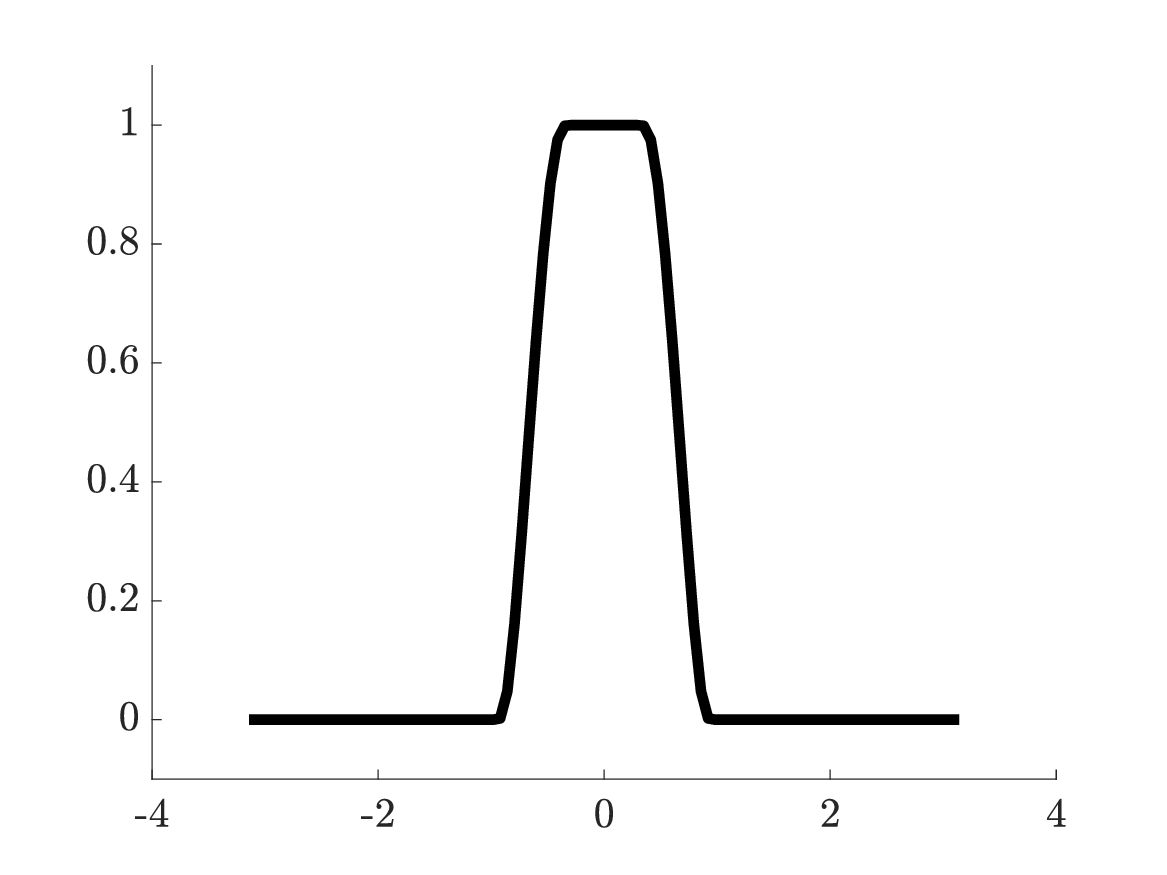}
\caption{$W_1$ as a function of $r\in[\pi,\pi]$.}
\end{subfigure}
\caption{Illustration of the window function $W_1$ with window widths $w_0 = 1/4$ and $w_1 = 1$, in the context of the domain $\Omega$ depicted in Figure~\eqref{fig:teardrop}. In this case, with singularity point $x=(2.5,1)$, the support of the window function intersects a smooth section of the boundary of $\Omega$. For other points $x\in\Omega$, the support may intersect a section of the boundary containing the corner point, or not intersect the boundary at all. }
\end{figure}

Our method for the evaluation of the integral $I_1(x)$, which is described in Section~\ref{sec:windowpart1}, hinges on a crucial generalization of Theorem~\ref{thm:krishna}, for the integral $\int_{\mathbb{R}^m} \chi_K(x)f(x)dx$---which broadens its scope to encompass integrals of the form $\int_{\mathbb{R}^m} \chi_K(x) g(x)f(x)dx$ including potentially singular functions $g$ such as $g(x) = \log|x|$. For simplicity we present this result in the case $m=1$ that is needed in the iterated integration strategy presented in this section.
\begin{lemma}
\label{mainlemma}
    Let $f$ be a periodic function on $[-\pi,\pi]$ of class $C^\infty$, and $g$ a periodic function on $[-\pi,\pi]$. Let $F\in\mathbb{N}$ and $N = 2F$. Denote the truncated Fourier series of $g$ by
\begin{equation}\label{trunc-log}
     g_F(x) = \sum_{n= -F}^F g_n \exp(\mathrm{i}{n}x).
\end{equation}
For any $p\in\mathbb{N}$, the following error estimate holds:
\begin{equation}
    \abs{\int_{-\pi}^\pi f(x) g(x)\ \dif x - Q_{N,[-\pi,\pi]} \left[f(x) g_F(x)\right] } = O(F^{-p}).
\end{equation}
\end{lemma}
\begin{proof}
A complete proof of a specific function $g$ is presented in \cite{br_pooj_2025}. In what follows we briefly outline a proof in the one dimensional case; the proof for higher dimensions follows along the lines of \cite{br_pooj_2025}.
We evaluate the integral
\begin{equation}
    \int_{-\pi}^\pi f(x) g(x)\ \dif x
\end{equation}
by substituting $g(x)$ by its Fourier series
\begin{equation}
     g(x) = \sum_{n= -\infty}^\infty g_n \exp(\mathrm{i}{n}x).
\end{equation}
The corresponding error is given by
\begin{equation}
\label{eq:estimate}
\begin{split}
    &\abs{\int_{-\pi}^\pi f(x) g(x)\ \dif x - Q_{N,[-\pi,\pi]} \left[f(x) g_F(x)\right] } \\
     \leq & \abs{\int_{-\pi}^\pi f(x) g(x)\ \dif x -\int_{-\pi}^\pi f(x) g_F(x)\ \dif x} + \abs{\int_{-\pi}^\pi f(x) g_F(x)\ \dif x-Q_{N,[-\pi,\pi]} \left[f(x) g_F(x)\right] }.
\end{split}
\end{equation}
The first term on the right-hand side of \eqref{eq:estimate} may be estimated as follows:
\begin{equation}
\label{eq:estimatepart1}
    \begin{split}
         & \abs{\int_{-\pi}^\pi f(x) g(x)\ \dif x -\int_{-\pi}^\pi f(x) g_F(x)\ \dif x} = \abs{\int_{-\pi}^\pi f(x) \sum_{\abs{n} >F} g_n \exp(\mathrm{i}{n}x) \ \dif x}\\
        =& \abs{\sum_{\abs{n} >F} g_n \int_{-\pi}^\pi f(x)  \exp(\mathrm{i}{n}x) \ \dif x} \leq  \sum_{\abs{n}>F} \abs{g_n}\abs{\frac{C_{f}}{n^{p+1}}} \leq C_1 \sum_{\abs{n}>F} \frac{1}{n^{p+1}}\\ \leq& C_2 \int_{F}^{\infty} \frac{1}{x^{p+1}} \ \dif x = C_3 F^{-p}.
    \end{split}
\end{equation}
where we use the estimate provided by Lemma \ref{lem1}. Using Lemma \ref{lem:trapz}, in turn, the second right-hand term in \eqref{eq:estimate} may be expressed in the form
\begin{equation}
    \label{eq:estimatepart2}
    \int_{-\pi}^\pi f(x) g_F(x)\ \dif x-Q_{N,[-\pi,\pi]} \left[f(x) g_F(x)\right] =  2\pi \sum_{k\in \mathbb{Z}_{\ne 0}} h_{kN},
\end{equation}
where $h(x) = f(x) g_F(x)$, with
\begin{equation}
    h_{kN} = \sum_{\abs{\ell}\leq F}f_{kN-\ell} {g}_{\ell}.
\end{equation}
Thus \eqref{eq:estimatepart2} becomes 
\begin{equation}
\begin{split}
     & \abs{2\pi \sum_{k\in\mathbb{Z}_{\ne 0}}\sum_{\abs{\ell}\leq F} g_lf_{kN-\ell}} 
     \leq  C_4 \sum_{k\in \mathbb{Z}_{\ne 0}} \sum_{\abs{\ell}\leq F} \frac{1}{\abs{kN-\ell}^p} \\
     \leq & C_5 \left(\sum_{k\in \mathbb{Z}_{>0}} \frac{1}{\abs{kN-F}^p} +  \sum_{k\in \mathbb{Z}_{<0}} \frac{1}{\abs{kN+F}^p}\right) \\
     \leq & C_6 \left( \sum_{k\in \mathbb{Z}_{>0}} \frac{1}{\abs{kN-F}^p} \right) 
    \leq C_7 \sum_{k\in\mathbb{Z}_{>0}} \frac{1}{\abs{ks-1}^p} \frac{1}{F^p} \leq C_8 F^{-p}.
\end{split}
\end{equation}

\end{proof}
\subsection{Computation of \texorpdfstring{$I_1(x)$}: windowed singularity}
\label{sec:windowpart1}
The integrand of the integral $I_1$ in~\eqref{I1I2} for a given $x=(x_1,x_2)\in \Omega$ has a singularity at $y=x$---which, as is well known, poses a significant challenge for numerical evaluation. 
To tackle this difficulty we first re-express the integral in polar coordinates. Letting $y-x = r(\cos(\theta),\sin(\theta))$ with $\theta\in [-\pi,\pi]$, the integral becomes
\begin{equation}
\label{eq:1.1polar}
    I_1(x)  = \int_{-\pi}^\pi \int_0^{d(\theta)} r \log \abs{r} \phi(x+r(\cos(\theta),\sin(\theta))) W_1(r) \ \dif r \dif \theta.
\end{equation}
Here, utilizing the window-width parameter $w_1$ introduced in \eqref{eq:windowfunction}, and employing the convention $\mathrm{dist}(x,\emptyset) = +\infty$, the distance function $d(\theta)$ is defined as
\begin{equation}
\label{eq:distancefunction}
    d(\theta) = \min(w_1, \mathrm{dist}(x,S_r(\theta)\cap \partial\Omega)).
\end{equation}
where $S_r(\theta)$ is defined as the line segment 
\begin{equation}
\label{eq:def_lines}
S_r(\theta)=\{x+(r\cos(\theta),r\sin(\theta))\ |\ r\in(0,w_1]\}. \\
\end{equation}

The method used to compute $I_1$ depends on how the support of the window function---namely, the disc $B_{w_1}(x) := \{\, y \mid \abs{y-x}\leq w_1 \,\}$ centered at $x$ with radius $w_1$---intersects the boundary $\partial\Omega$ of the domain $\Omega$. In this section, we discuss the three possible cases, depicted in Figure~\eqref{fig:threeintersections_illustration}, namely (i)~$B_{w_1}(x)\cap \partial\Omega = \emptyset$; (ii)~$B_{w_1}(x)\cap \partial\Omega$ is non-empty and it contains no corner points; and, (iii)~$B_{w_1}(x)\cap \partial\Omega$ contains a corner point. 
\begin{figure}[ht]
\centering
\begin{subfigure}[t]{0.32\textwidth}
    \centering
    \includegraphics[width=\linewidth]{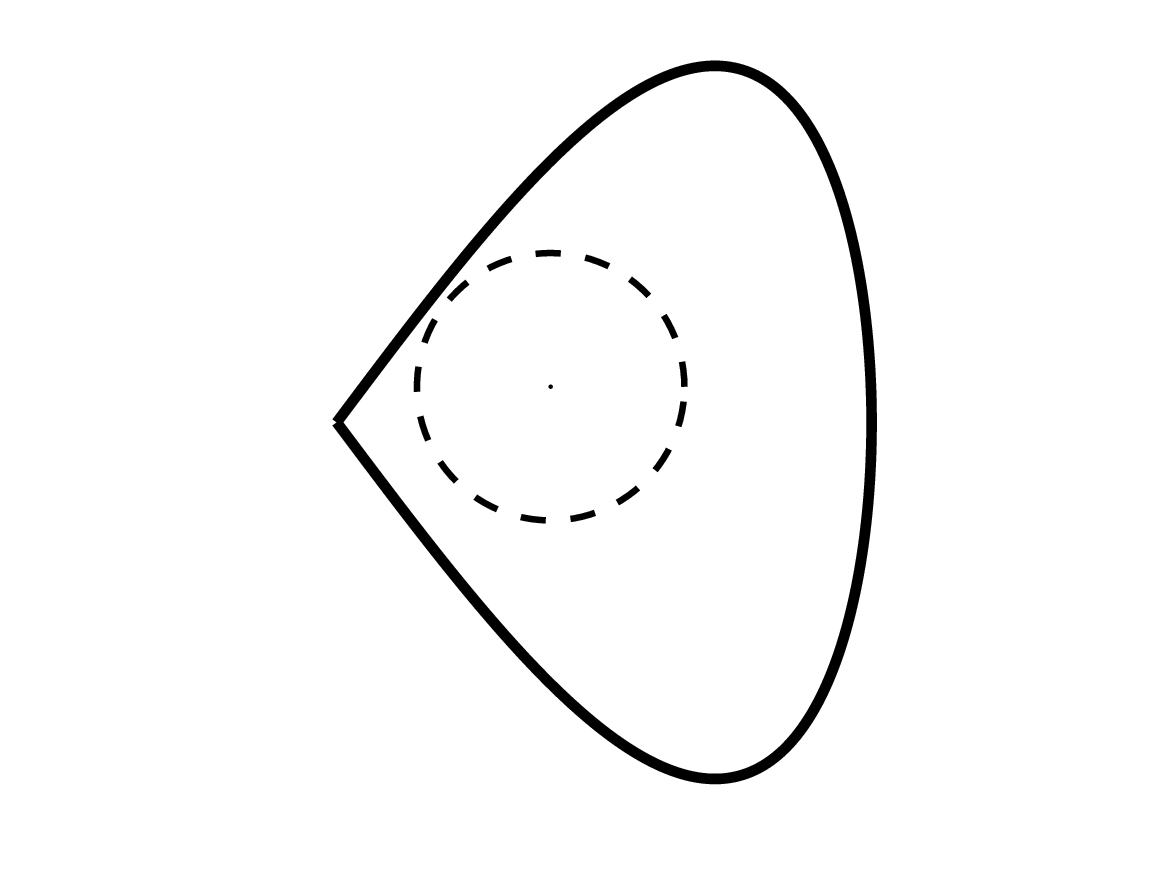}
     \caption{$B_{w_1}(x)\cap \partial\Omega = \emptyset$.}
     \label{fig:6-1}
\end{subfigure}
\begin{subfigure}[t]{0.32\textwidth}
    \centering
    \includegraphics[width=\linewidth]{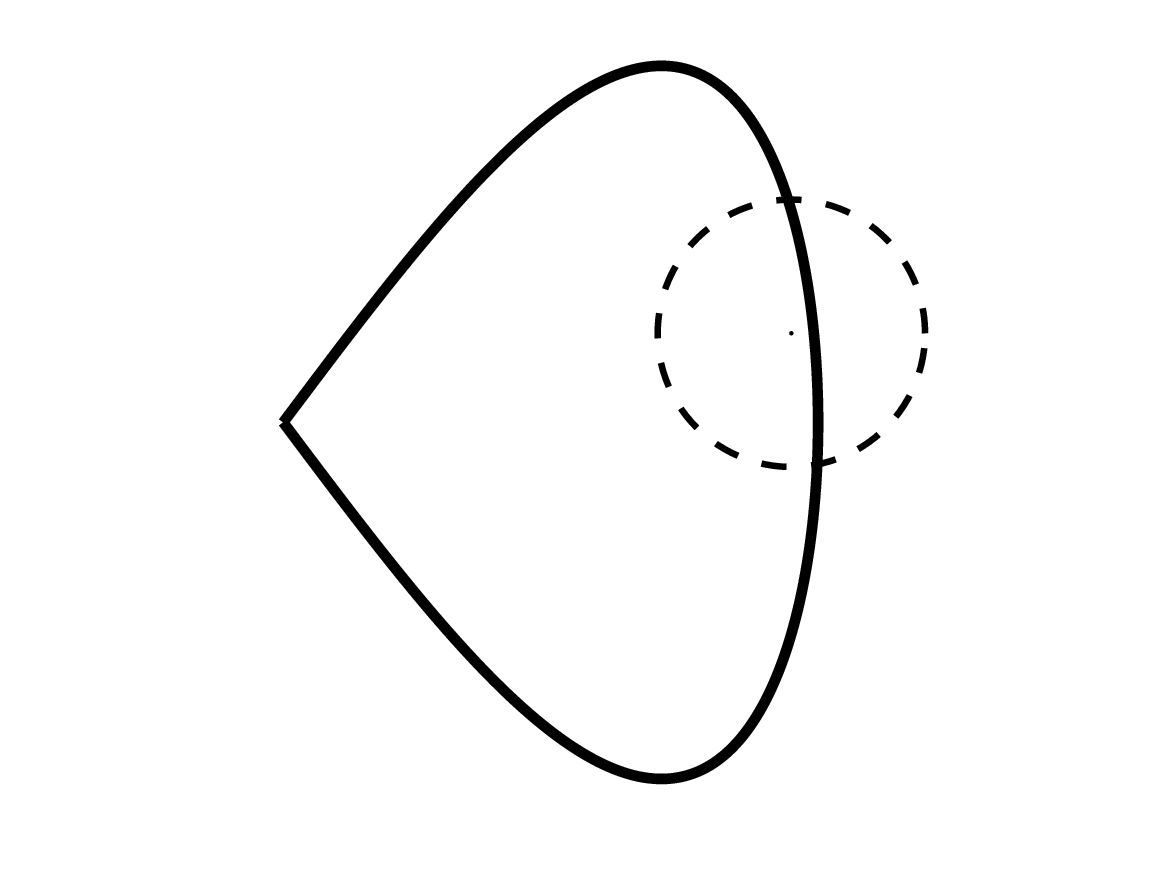}
     \caption{$B_{w_1}(x)\cap \partial\Omega \ne \emptyset$ and the intersection contains no corner points.}\label{fig:6-3}
\end{subfigure}
\begin{subfigure}[t]{0.32\textwidth}
    \centering
    \includegraphics[width=\linewidth]{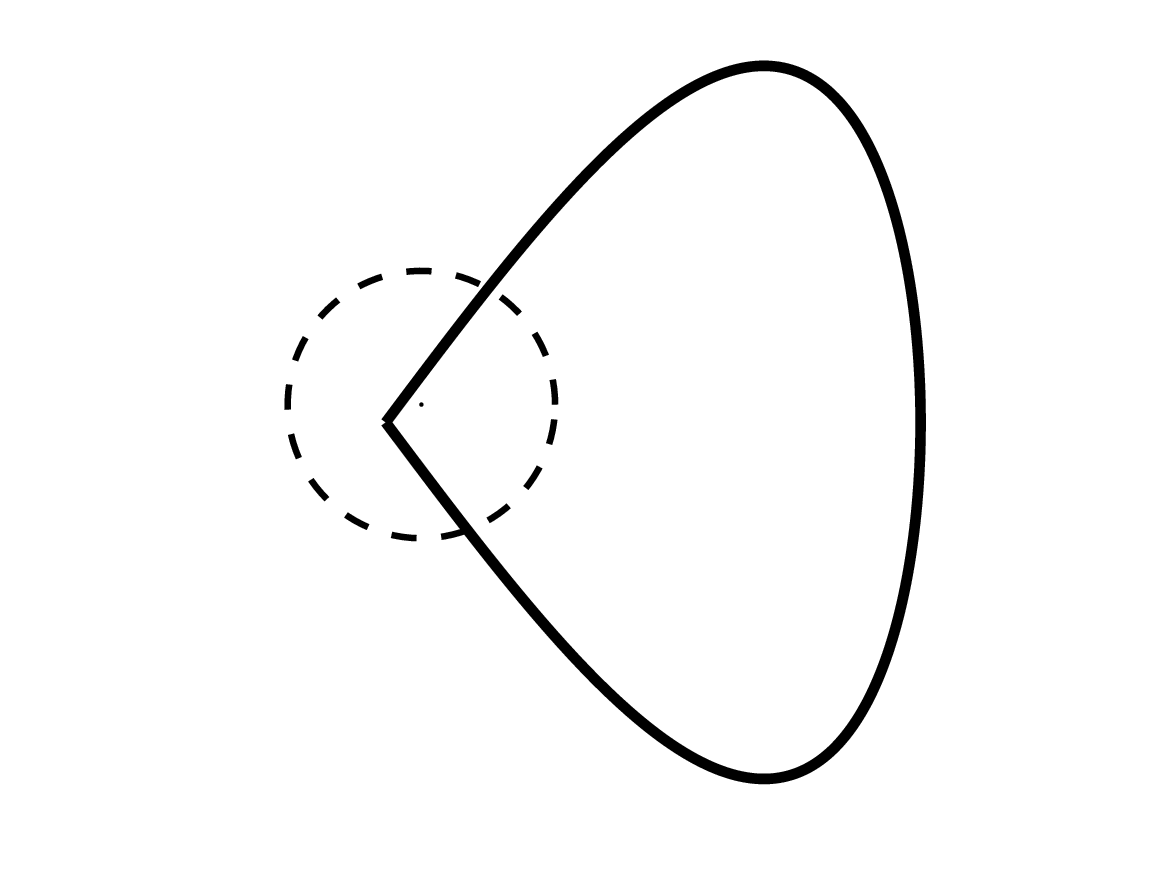}
     \caption{$B_{w_1}(x)\cap \partial\Omega \ne \emptyset$ contains a corner point.}\label{fig:6-5}
\end{subfigure}
    \caption{Three intersection cases considered in this section.}
    \label{fig:threeintersections_illustration}
\end{figure} 

We first consider the situation depicted in Figure \ref{fig:6-3} and described in point~(ii) above, for which $d(\theta)$ is a continuous but non-smooth function, with discontinuous first order derivatives. (The case of point~(i) may be regarded as the special situation where $d(\theta)$ is constant in $\theta$.)
In order to design a numerical algorithm based on Lemma \ref{mainlemma} for the evaluation of the integral $I_1$, in the present case we consider the equivalent expression
\begin{equation}
\label{eq:part1method}
        I_1(x)  = \int_{0}^\pi \int_{-d(\theta+\pi)}^{d(\theta)} \abs{r} \log
        \abs{r} \phi(x+r(\cos(\theta),\sin(\theta))) W_1(r)\ \dif r \dif \theta.
\end{equation}
Furthermore, letting
\begin{equation}
\label{eq:log-chi}
        L(r,\theta) = \abs{r}\log\abs{r} \chi_{[-d(\theta+\pi),d(\theta)]}
\end{equation}
and defining
\begin{equation}
\label{eq:I1r}
    I_{1,x}(\theta) = \int_{-w_1}^{w_1} L(r,\theta) \phi(x+r(\cos(\theta),\sin(\theta))) W_1(r) \ \dif r,
\end{equation}
we re-express $I_1$ in the form
\begin{equation}
\label{eq:intI1x}
    I_1(x) = \int_{0}^\pi I_{1,x}(\theta) \dif\theta.
\end{equation}

To proceed, we seek to apply Lemma \ref{mainlemma} for the computation of $I_{1,x}(\theta)$. To this end, we define the truncated Fourier series $L^{F}(r,\theta)$ of the function $L(r,\theta)$, viewed as a periodic function in the $r$-variable over the interval $[-w_1, w_1]$.
    \begin{equation}
    \label{eq:def_of_LF}
        L(r,\theta) = \sum_{n= -\infty}^\infty L_n(\theta) e^{(\mathrm{i}\frac{2\pi n}{P}r)}, \quad L^F(r,\theta) = \sum_{n= -F}^{F-1} L_n(\theta) e^{(\mathrm{i}\frac{2\pi n}{P}r)}\quad \text{with}  P = 2w_1.
    \end{equation}
    where
    \begin{equation}
        L_n(\theta)  = \frac{1}{P} \int_0^{d(\theta)} \log(t) t \exp(-\mathrm{i}2\pi\frac{n}{P}t)\  \mathrm{d}t.
    \end{equation}

The series involves Fourier modes $L_n$ with indices $-F \leq n \leq F-1$, and explicit formulas for the coefficients $L_n$ are provided in Appendix~\ref{sec:Cnrlogr}. By Lemma \ref{mainlemma}, applying the trapezoidal rule \eqref{eq:trapz} to
\begin{equation}
\label{eq:part1method-2}
        \int_{-w_1}^{w_1} L^{F}(r,\theta) \phi(x+r(\cos(\theta),\sin(\theta))) W_1(r)  \ \dif r
\end{equation}
with $N = 2F$  yields superalgebraic convergence to $I_{1,x}(\theta)$. 

Since the window function is smooth and vanishes at the boundary, and the intersection of its support with the boundary is also smooth, we have the following lemma.
\begin{lemma}
\label{lem:I1x}
    Let $x\in\Omega$ be such that $B_{w_1}(x)\cap \partial\Omega$ contains no corner points, where $w_1>0$ denotes the width of the window function~\eqref{eq:windowfunction}. Then, the  integral $I_{1,x}(\theta)$ is an infinitely differentiable  periodic function of $\theta$ of periodicity $\pi$.
\end{lemma}
\begin{proof}
   Defining
   \begin{equation}
        \widetilde{I}_{1,x}(\theta): = \int_0^{d(\theta)} \log \abs{r} \phi(x+r(\cos(\theta),\sin(\theta))) W_1(r)r \ \dif r,
    \end{equation}
  and in view of the relation
    \begin{equation}\label{eq:theta_p_pi}
        I_{1,x}(\theta) = \widetilde{I}_{1,x}(\theta) + \widetilde{I}_{1,x}(\theta+\pi)
    \end{equation}
to establish the lemma it suffices to show that $\widetilde{I}_{1,x}(\theta)$ is smooth function of $\theta$ for $0\leq \theta\leq 2\pi$. (The $\pi$-periodicity of $I_{1,x}(\theta)$ follows from the $2\pi$-periodicity of $d(\theta)$ together with equation~\eqref{eq:theta_p_pi}.)
To do this, let $\theta_i\in  [0,2\pi)$, $i=1,\dots,p$, denote the values of $\theta$ that correspond to the points of intersection of $\partial B_{w_1}(x)$ and $\partial\Omega$ at which $\partial B_{w_1}(x)$ and $\partial\Omega$ are not mutually tangent. Clearly, the distance function $d(\theta)$ is smooth on $[0,2\pi) \setminus \{\theta_i\}_{i=1,\dots,p}$, and, thus, so is $\widetilde{I}_{1,x}(\theta)$. It therefore suffices to show that $\widetilde{I}_{1,x}(\theta)$ is infinitely differentiable at the intersection angles $\theta_i$---at which $d(\theta)$ is continuous but not differentiable, although it is infinitely differentiable from the left and from the right. 

 Letting
 \begin{equation}
        \psi(r,\theta) : = \log \abs{r} \phi(x+r(\cos(\theta),\sin(\theta))) W_1(r)r,
    \end{equation}
 to establish the infinite differentiability of $\widetilde{I}_{1,x}(\theta)$ at $\theta = \theta_i$ we first note  that, since $d(\theta_i) = w_1$ and since, by definition, $W_1$ and all its derivatives vanish at $r = w_1$, it follows $\psi(r,\theta)$ and all its derivatives vanish at $r = d(\theta_i)$. Thus, to complete the proof of the lemma it suffices to establish the following statement:
\begin{quote}
     ``Let $n\in\mathbb{N}$. Then, for every function $\psi(r,\theta)$ that is infinitely differentiable in a neighborhood of $(d(\theta_i),\theta_i)$ with all its derivatives vanishing at $(d(\theta_i),\theta_i)$, the integrated function $\int_0^{d(\theta)} \psi(r,\theta) \dif r$ is $n$-times continuously differentiable at $\theta = \theta_i$ and 
     \begin{equation}\label{eq:ind-form}
        \frac{d^{n}}{d\theta^{n}}  \int_0^{d(\theta)} \psi(r,\theta) \dif r  = \int_0^{d(\theta)} \frac{\partial^n}{\partial \theta^n} \psi(r,\theta) \dif r+  \Theta_n(\theta),
     \end{equation}
    where $\Theta_n(\theta)$ is infinitely differentiable in a neighborhood of $\theta_i$ and all its derivatives vanish at $\theta = \theta_i$.\mbox{''}
\end{quote}
     The proof of this statement follows by induction. In the case of $n=0$, $\Theta_0(\theta) =0$, and the statement holds in view of the continuity of $d(\theta)$ at $\theta = \theta_i$. Then, assuming the statement is true for $n\in \mathbb{N}$, we have 
     \begin{equation}
        \frac{d^{n+1}}{d\theta^{n+1}}  \int_0^{d(\theta)} \psi(r,\theta) \dif r  =  \frac{d}{d\theta} \int_0^{d(\theta)} \frac{\partial^n}{\partial \theta^n} \psi(r,\theta) \dif r + \frac{d}{d\theta} \Theta_n(\theta).
     \end{equation}
     Employing the Leibniz integral differentiation rule we then obtain
     \begin{equation}
        \frac{d}{d\theta} \int_0^{d(\theta)} \frac{\partial^n}{\partial \theta^n} \psi(r,\theta) \dif r  =  \int_0^{d(\theta)} \frac{\partial^{n+1}}{\partial \theta^{n+1}} \psi(r,\theta) \dif r+  \frac{\partial^n}{\partial \theta^n} \psi(d(\theta),\theta)\frac{\dif}{\dif \theta} d(\theta).
     \end{equation}
     It follows that
     \begin{equation}
         \frac{d^{n+1}}{d\theta^{n+1}}  \int_0^{d(\theta)} \psi(r,\theta) \dif r   =  \int_0^{d(\theta)} \frac{\partial^{n+1}}{\partial \theta^{n+1}} \psi(r,\theta) \dif r+  \frac{\partial^n}{\partial \theta^n} \psi(d(\theta),\theta) \frac{\dif}{\dif \theta} d(\theta) +  \frac{d}{d\theta}\Theta_n(\theta),
     \end{equation}
     which coincides with~\eqref{eq:ind-form} provided
     \begin{equation}
        \Theta_{n+1}(\theta) :=  \frac{\partial^n}{\partial \theta^n} \psi(d(\theta),\theta)\frac{\dif}{\dif \theta} d(\theta) + \frac{d}{d\theta} \Theta_n(\theta).     
     \end{equation}
     The infinitely differentiability of $\Theta_{n+1}(\theta)$, with vanishing derivatives at $\theta = \theta_i$ follows by inductive hypothesis, since $d(\theta)$ is infinitely differentiable from the left and from the right at $\theta = \theta_i$ and since, by hypothesis, all of the derivatives of $\psi(r,\theta)$ vanish at that point. The proof is thus complete.\end{proof}
     
     
     

In view of Lemma~\ref{lem:trapz}, the trapezoidal rule in the variable $\theta$ evaluates the integral~\eqref{eq:intI1x} with superalgebraically small errors. Together with Lemma~\ref{mainlemma}, which ensures accurate evaluation of $I_{1,x}(\theta)$ for each $\theta$, this leads to the superalgebraically convergent TFF-based quadrature rule for the 2D integral~\eqref{eq:part1method}---as stated in Theorem~\ref{thm:quadratureI_smooth}. In detail, using~\eqref{eq:def_of_LF} to define the function
\begin{equation}
\label{eq:Phi}
\Phi^F: \mathbb{R}^2 \to \mathbb{C}, \quad
\Phi^F(r,\theta) = L^{F}(r,\theta),\phi\big(x+r(\cos\theta,\sin\theta)\big),W_1(r)
\end{equation}
and employing the notation~\eqref{eq:2dtrapz} for the two-dimensional trapezoidal quadrature rule, the proposed TFF quadrature rule for the integral $I_1(x)$---in the case where $B_{w_1}(x)\cap \partial\Omega$ contains no corner points---is given by
\begin{equation}
        \label{eq:quadratureI_smooth}
        Q_{N,\left[0,\pi\right]}^2\left[Q_{N,\left[-w_1,w_1\right]}^1\left[\Phi^F\right]\right].
\end{equation}
The convergence properties of this quadrature rule are established in the following theorem.
\begin{theorem}
    \label{thm:quadratureI_smooth}
    Let $F \in \mathbb{N}$ and $N = 2F$. and $x\in\Omega$, and assume $B_{w_1}(x)\cap \partial\Omega$ contains no corner points. Then, the TFF quadrature rule~\eqref{eq:quadratureI_smooth} converges superalgebraically fast to $I_1(x)$.
\end{theorem}
\begin{proof}
The error of the TFF quadrature is given by
\begin{equation}
\label{eq:estimate2}
\begin{split}
      & \abs{I_1(x) - Q_{N,\left[0,\pi\right]}^2\left[Q_{N, \left[-w_1,w_1\right]}^1\left[\Phi^F(r,\theta)\right]\right]} \\
    = & \abs{\int_0^\pi I_{1,x}(\theta) \dif \theta - Q_{N,\left[0,\pi\right]}^2\left[Q_{N, \left[-w_1,w_1\right]}^1\left[\Phi^F(r,\theta)\right]\right]} \\
    = & \abs{\int_0^\pi I_{1,x}(\theta) \dif \theta - Q_{N,\left[0,\pi\right]}^2\left[Q_{N, \left[-w_1,w_1\right]}^1\left[\Phi^F(r,\theta)\right]- I_{1,x}(\theta)+ I_{1,x}(\theta)\right]} \\
    = & \abs{\int_0^\pi I_{1,x}(\theta) \dif \theta - Q_{N,\left[0,\pi\right]}^2\left[I_{1,x}(\theta)\right] - Q_{N,\left[0,\pi\right]}^2\left[Q_{N, \left[-w_1,w_1\right]}^1\left[\Phi^F(r,\theta)\right]- I_{1,x}(\theta)\right]} \\
    \leq & \abs{\int_0^\pi I_{1,x}(\theta) \dif \theta - Q_{N,\left[0,\pi\right]}^2\left[I_{1,x}(\theta)\right]} +  \abs{Q_{N,\left[0,\pi\right]}^2\left[I_{1,x}(\theta)-Q_{N, \left[-w_1,w_1\right]}^1\left[\Phi^F(r,\theta)\right]\right]}
\end{split}
\end{equation}
By Lemma \ref{lem:trapz}, for the first term of on the right hand of this equation we have
\begin{equation}
\label{eq:geo}
    \abs{\int_0^\pi I_{1,x}(\theta) \dif \theta - Q_{N,\left[0,\pi\right]}^2\left[I_{1,x}(\theta)\right]}   \leq \pi\left | \sum_{j\in N\mathbb{Z}\setminus \{0\}}\eta_{j}\right| ,
\end{equation}
where $\eta_{j}$ is the $j$-th Fourier coefficient of $I_{1,x}(\theta)$. But, in view of the smoothness and periodicity of the function $I_{1,x}(\theta)$ established in Lemma \ref{lem:I1x}, the Fourier coefficients $\eta_{j}$ decay superalgebraically fast---that is, for  any $q\in\mathbb{N}$ we have $|\eta_j|\leq C_q j^{-q}$ for some constant $C_q$. It follows that, for any $q\in\mathbb{N}$, the quantity~\eqref{eq:geo} is bounded by $C_q N^{-q}\sum_{i =1}^\infty i^{-q}$ and the first term on the right hand side therefore tends to zero superalgebraically fast. Concerning the second right-hand term we first note that, by Lemma \ref{mainlemma}, we have the bound
\begin{equation}
   \abs{ I_{1,x}(\theta)-Q_{N, [-w_1,w_1]}^1\left[\Phi^F(r,\theta)\right]} = O(F^{p+1})
\end{equation}
for any $p\in\mathbb{N}$. The application of the trapezoidal rule with $N = 2F$ points can only increase the error by a factor proportional to $N$, and, thus the second right-hand term is superalgebraically small as well. The proof is thus complete.
\end{proof}

\begin{figure}[ht]
    \centering
\includegraphics[width=\linewidth]{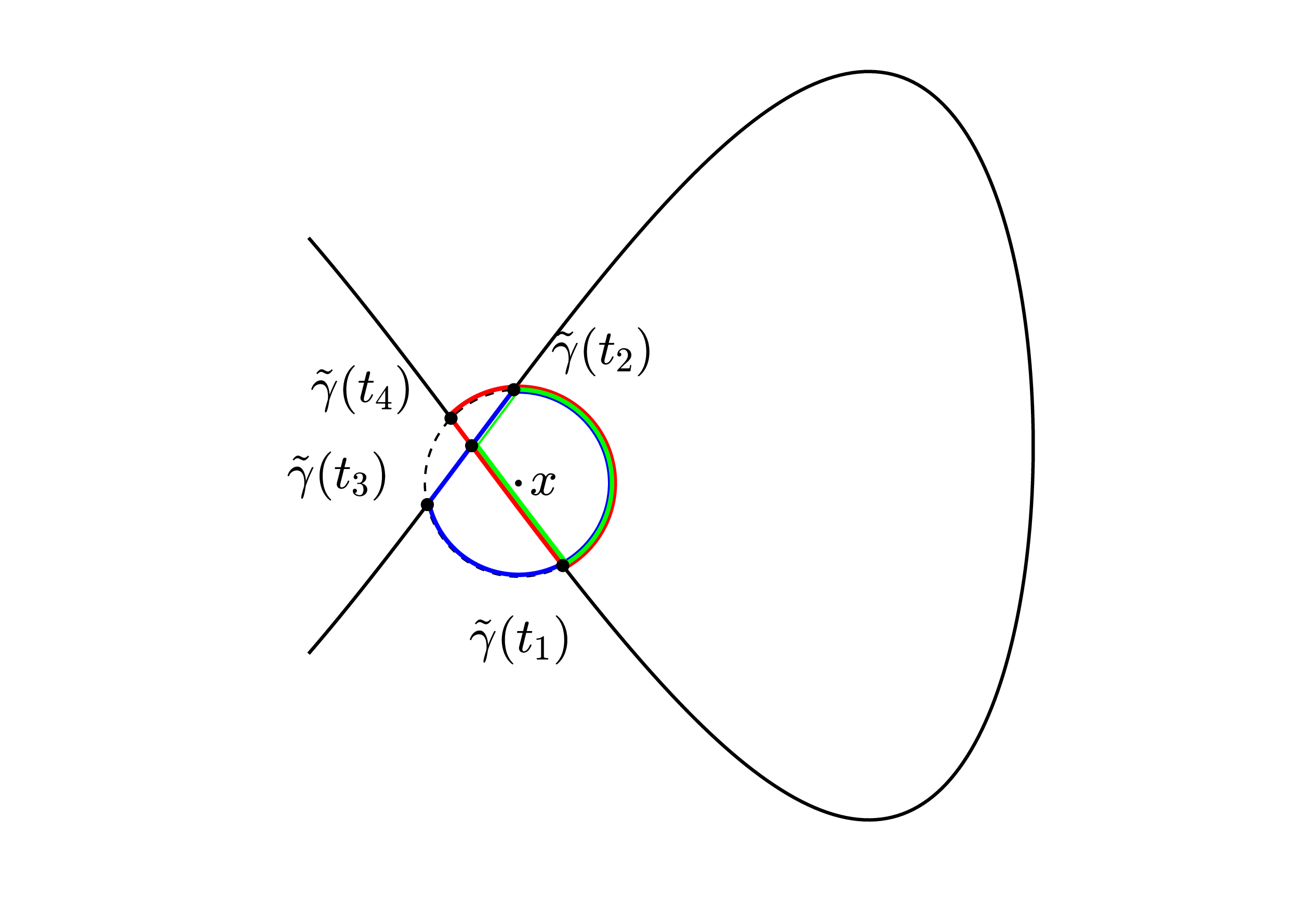}
    \caption{Illustration of the method used when the window function intersects with a non-smooth corner (case (iii))}
    \label{fig:method3}
\end{figure}

We now consider the case where $B_{w_1}(x)\cap \partial\Omega$ contains a corner point of $\partial\Omega$. In this situation, the integral $I_{1,x}(\theta)$ is not a smooth function of $\theta$ at the angular value $\theta$ corresponding to the corner point (since, at that point, $d(\theta)$ is not differentiable and the window function does not vanish), and, therefore, Lemma~\ref{lem:I1x} does not directly apply in this case. Consequently,  a direct application of the trapezoidal rule to integrate $I_{1,x}(\theta)$ in the $\theta$ variable results in slow convergence.

We propose to restore fast convergence by incorporating TFF-based quadrature in the {\em angular} direction, in addition to the radial TFF quadrature already employed. To do this we utilize a curve $\widetilde\Gamma$ that extends $\partial\Omega$ past the corner point, as illustrated in Figure~\ref{fig:method3}, by employing the known curve parameterization. For the particular curve considered in this section, we define 
\begin{equation}
\widetilde{\gamma}(t):=(l_x\sin(t/2),-l_y\sin(t)),\qquad 
    \widetilde{\Gamma} = \{\widetilde{\gamma}(t)\ |\ t\in [-\delta,2\pi+\delta]\},
\end{equation}
where $\delta$ is selected appropriately, so that the circle $B_{w_1}(x)$ intersects  $\widetilde{\Gamma}$ at four points $\widetilde{\gamma}(t_1)$, $\widetilde{\gamma}(t_2)$, $\widetilde{\gamma}(t_3)$ and $\widetilde{\gamma}(t_4)$ with $t_1,t_2\in [0,2\pi]$ and $t_3\in [2\pi,2\pi+\delta]$ and $t_4\in [-\delta,0]$. The parameter value corresponding to the corner point, in turn, is denoted by $t_0$ and the angular coordinate of the point $\widetilde{\gamma}(t_0)$ in the polar coordinates system centered at $x$ is denoted by $\theta^x_0\in(-\pi,\pi]$. Further, for $j=1,2,3,4$, let $\theta_j^x$ denote the angular coordinate of the point $\widetilde{\gamma}(t_j)$ with angular range $\theta_0^x \leq \theta_j^x \leq \theta_0^x + 2\pi$. 

Using this notation,  to facilitate the application of the TFF method in this context we note that the inner integral in~\eqref{eq:1.1polar} is a $2\pi$-periodic function of $\theta$, and that, therefore $I_{1}(x)$ can be re-expressed in the form
\begin{equation}
\label{eq:I_1theta0}
I_{1}(x)   = \int_{\theta_0^x}^{\theta_0^x+2\pi} \dif \theta \int_{0}^{d(\theta)} r \log \abs{r} \phi(x+r(\cos(\theta),\sin(\theta))) W_1(r) \  \dif r,
\end{equation}
and, further,
\begin{equation}\label{eq:sum12}
     I_1(x)  = {I_{1}^1}(x)  + {I_{1}^2}(x),
\end{equation}
where
\begin{equation}\label{eq:red_blue}
\begin{split}
    & {I_{1}^1}(x)   = \int_{\theta_0^x}^{\theta_0^x+2\pi} \chi (\theta;\theta_0^x,\theta_2^x) \dif \theta\int_{0}^{d(\theta)} r \log \abs{r} \phi(x+r(\cos(\theta),\sin(\theta))) W_1(r) \ \dif r,\\
    & {I_{1}^2}(x) = \int_{\theta_0^x}^{\theta_0^x+2\pi} (1-\chi (\theta;\theta_0^x,\theta_2^x)) \dif \theta\int_{0}^{{d(\theta)}} r \log \abs{r} \phi(x+r(\cos(\theta),\sin(\theta))) W_1(r) \ \dif r .
\end{split}
\end{equation}
This setup is depicted in Figure~\ref{fig:method3}, wherein the upper limit of integration along the radial coordinate in the inner integrals in~\eqref{eq:I_1theta0} and~\eqref{eq:red_blue} is represented in green. 

Unfortunately the $\theta$ integrals in~\eqref{eq:red_blue} cannot be treated directly by means of the TFF algorithm since, although they are expressed in terms of characteristic functions, the accompanying function (namely, the inner integrals in~\eqref{eq:red_blue}) are not smooth functions of $\theta$ at the corner point $\theta =\theta_0^x$---since the function $d(\theta)$ is not smooth at that point. 

To bypass this difficulty we first re-express the integrals~\eqref{eq:red_blue} in terms of integrals with upper limits of integration represented by the red and blue in Figure~\ref{fig:method3}. In detail, defining 
\begin{equation}
    \widetilde{\Gamma}_{(a,b)}= \{\widetilde{\gamma}(t) = (l_x\sin(t/2),-l_y\sin(t))| t\in (a,b)\},\quad a,b\in\mathbb{R},\quad a < b,
\end{equation}
using $S_r(\theta)$~\eqref{eq:def_lines},  letting $d_1(\theta)$ and $d_2(\theta)$ denote  the distance between $x$ and the red and blue curves along the $\theta$ angular direction, that is to say,
\begin{equation}
\label{eq:fred}
    d_{1}(\theta) = 
    \begin{cases}
        \mathrm{dist}(x,S_r(\theta)\cap  \widetilde{\Gamma}_{(t_4,t_1)}), & \theta \in (\theta_4^x,\theta_0^x) \\
        w_1  & \text{otherwise},\\
        \end{cases}
\end{equation}
and
\begin{equation}
\label{eq:fblue}
    d_{2}(\theta) = 
    \begin{cases}
        \mathrm{dist}(x,S_r(\theta)\cap  \widetilde{\Gamma}_{(t_3,t_2)}), & \theta \in (\theta_2^x,\theta_3^x), \\ 
        w_1 & \textrm{otherwise},
    \end{cases}
\end{equation}
and employing the characteristic function $\chi (\theta;\theta_0^x,\theta_2^x):= \chi_{[\theta_0^x,\theta_2^x]}(\theta)$, we re-express the integrals~\eqref{eq:red_blue} in the form
\begin{equation}\label{eq:red_blue_new}
\begin{split}
    & {I_{1}^1}(x)   = \int_{\theta_0^x}^{\theta_0^x+2\pi} \chi (\theta;\theta_0^x,\theta_2^x) \dif \theta \int_{0}^{d_1(\theta)} r \log \abs{r} \phi(x+r(\cos(\theta),\sin(\theta))) W_1(r) \ \dif r.\\
    & {I_{1}^2}(x) = \int_{\theta_0^x}^{\theta_0^x+2\pi} (1-\chi (\theta;\theta_0^x,\theta_2^x)) \dif \theta \int_{0}^{{d_2(\theta)}} r \log \abs{r} \phi(x+r(\cos(\theta),\sin(\theta))) W_1(r) \ \dif r.
\end{split}
\end{equation}
Letting, further
\begin{equation}
\label{eq:log-chi_redblue}
        L_{1}(r,\theta) = \abs{r}\log\abs{r} \chi_{[0,d_1(\theta)]}, \quad
        L_{2}(r,\theta) = \abs{r}\log\abs{r} \chi_{[0,d_2(\theta)]}
\end{equation}
we may re-express \eqref{eq:red_blue_new} as
\begin{equation}\label{eq:red_blue2}
\begin{split}
    & I_1^1(x)  = \int_{-\pi}^\pi \chi (\theta;\theta_0^x,\theta_2^x) \dif \theta  \int_{-w_1}^{w_1} L_{1}(r,\theta) \phi(x+re^{\ii \theta}) W_1(r)\ \dif r\\
    & I_1^2 (x) = \int_{-\pi}^\pi (1-\chi (\theta;\theta_0^x,\theta_2^x)(\theta) ) \dif \theta \int_{-w_1}^{w_1} L_{2}(r,\theta) \phi(x+re^{\ii \theta}) W_1(r)\ \dif r.
\end{split}
\end{equation}
In view of Lemma~\ref{mainlemma}, fast numerical quadrature for the radial integral in~\eqref{eq:red_blue2} is obtained by substituting the functions $L_1(r,\theta)$ and $L_2(r,\theta)$ in the corresponding integrands by the Fourier series expansions
\begin{equation}
    L_{1}^F(r,\theta) = \sum_{n= -\infty}^\infty L_{1,n}(\theta) e^{\mathrm{i}\frac{2\pi n}{P}r}\quad\mbox{and}\quad L_{2}^F(r,\theta) = \sum_{n= -\infty}^\infty L_{2,n}(\theta) e^\frac{2\pi n}{P}r,\quad \mbox{with}\quad P = 2w_1.
\end{equation}
and then applying the trapezoidal-rule, in each case, using $N=2F$ discretization points.
In view of Lemma~\ref{lem:I1x}, further, the inner integrals in~\eqref{eq:red_blue_new}, and, equivalently, the inner integrals in~\eqref{eq:red_blue2}, are smooth functions of $\theta$. We can thus apply Lemma~\ref{mainlemma} once again: by replacing $\chi (\theta;\theta_0^x,\theta_2^x)$ by its truncated Fourier series $\chi^F (\theta;\theta_0^x,\theta_2^x)$, in~\eqref{eq:red_blue2},  fast trapezoidal-rule convergence in the angular integral  is achieved provided a number $N_\theta = 2F_\theta$ of discretization points is used. 

In sum, letting
\begin{equation}
    \Phi_i^F(r,\theta) : = L_i^{F}(r,\theta) \phi(x+r(\cos(\theta),\sin(\theta))) W_1(r),\quad \text{for } i=1,2
\end{equation}
(cf.~\eqref{eq:Phi}), we propose
the corner-point capable TFF quadrature rule for the integral $I_1(x)$, namely
\begin{equation}
        \label{eq:quadratureI_corner}
        I_1(x)\approx Q_{N,[-\pi,\pi]}^2[(1-\chi^F (\theta;\theta_0^x,\theta_2^x))Q_{N, [-w_1,w_1]}^1[\Phi_1^F ]] + Q_{N,[-\pi,\pi]}^2[\chi^F (\theta;\theta_0^x,\theta_2^x) Q_{N, [-w_1,w_1]}^1[\Phi_2^F ]].
\end{equation}

We thus obtain the following theorem, whose proof, which is omitted for brevity, follows analogously to the proof of Theorem~\ref{thm:quadratureI_smooth}.
\begin{theorem}        \label{thm:quadratureI_corner}
    Let $F \in \mathbb{N}$ and $N = 2F$. For $x\in\Omega$, if $B_{w_1}(x)\cap \partial\Omega$ contains a non-smooth corner point, the TFF quadrature converges superalgebraically to $I_1(x)$.
\end{theorem}

\subsection{Computation of \texorpdfstring{$I_2(x)$}: smooth integrand.} 
\label{sec:windowpart2}
This section presents an algorithm for the fast and accurate evaluation of the second part of the convolution integral~\eqref{eq:splitmainintegral}, namely
\begin{equation}
\label{eq:I2}
    I_2(x) = \int_\Omega \log \abs{x-y}\phi(y)(1-W_1(x-y))\mathrm{d}y,
\end{equation}
for all $x\in\Omega$, by means of the FFT algorithm. By employing a translation, if necessary, we assume that the smallest rectangle containing $\Omega$ is centered at the origin, and we introduce a larger zero-centered rectangle
\begin{equation}
\label{eq:sqaureR}
R := [-P_1/2,P_1/2] \times [-P_2/2,P_2/2] \subset \mathbb{R}^2    
\end{equation}
containing $\Omega$ and, we let $\delta >0$ denote the distance between $\partial \Omega$ and $\partial R$. The condition $w_1 < \delta$  guarantees that the support of $W_1(x-y)$ is contained within $R$ for all $y\in \Omega$; in practice  values $P_1$ and $P_2$ (and, thus, of $\delta$) are selected so as to balance the overall size of the rectangle $R$ and the sharpness of the window function $W_1$---since both, large rectangles and sharp window functions lead to the requirement of large numbers of discretization points for a given accuracy. 

Letting $F \in \mathbb{N}$ and $N = 2F$, we obtain the discretization 
\begin{equation}
    y_{k,\ell} = \Bigl(k\frac{P_1}{N},\,\ell\frac{P_2}{N}\Bigr), 
    \qquad k,\ell=-F,\dots,F-1
\end{equation}
of the rectangle $R$. Using methods mentioned in \ref{sec:divergence} we can compute the 2D Fourier coefficients of the characteristic function of $\chi_\Omega$ in the periodicity domain $R$ accurately. Then, calling $\chi_\Omega^{j,k}$ these Fourier coefficients, denoting by
\begin{equation}
\label{eq:fourierchar}
    \chi_\Omega^F(y) = \sum_{j,k=-F}^F \chi_\Omega^{j,k}\,
    e^{\ii j \tfrac{2\pi}{P_1}y_1}\,
    e^{\ii k \tfrac{2\pi}{P_2}y_2}
\end{equation}
the $2F+1$-term truncated Fourier expansion of $\chi_\Omega$, employing the two-dimensional trapezoidal-rule weights 
$w_{k,\ell}$ (equal to one for nodes interior to $R$, 
equal to one half along the edges, 
and one quarter at the corners of $R$), and letting
\begin{equation}
    f^F(y) = \chi_\Omega^F(y)\phi(y),
    \qquad 
    g(y) = \log\abs{y}\,(1-W_1(y)), 
\end{equation}
in view of Theorem \ref{thm:krishna}, we obtain the superalgebraically convergent approximation
\begin{equation}
\label{eq:convsum}
 I_2(y_{m,n}) =    \frac{P_1P_2}{N^2} 
       \sum_{k,\ell=-F}^{F-1} w_{k,\ell}\, f^F(y_{k,\ell})\,g(y_{m,n}-y_{k,\ell}) 
    = \frac{P_1P_2}{N^2} 
       \sum_{k,\ell=-F}^{F-1} w_{k,\ell}\, f^F(y_{k,\ell})\,g(y_{m-k,n-\ell}),
\end{equation}
 of the integral  \eqref{eq:I2} at $x=y_{m,n}$, where the second sum explicitly displays the dependence of the kernel $g$ on $y_{m-k,n-\ell}$. 
 
Clearly, the right-hand sum in~\eqref{eq:convsum} equals the discrete convolution of the sequences 
\[
\{\,w_{k,\ell} f(y_{k,\ell})\,\} 
\quad\text{and}\quad 
\{\,g(y_{k,\ell})\,\}.
\]
By invoking the convolution theorem, this convolution can be computed efficiently by applying the FFT and inverse FFT with zero padding \cite{rao2011fast}.
\subsection{Algorithm and Pseudocode}
The pseudocode presented in Algorithm~\ref{alg:alg_1} incoporates the numerical strategies described in Sections~\ref{sec:windowpart1} and \ref{sec:windowpart2}
\begin{algorithm}
\caption{Truncated Fourier Filtering for Singular-Kernel Convolution}
\label{alg:alg_1}
\begin{algorithmic}[1]

\State Construct the equidistant convolution-integral evaluation grid $\{x_{i,j}\}$ on the rectangle $R$ (eqs.~\eqref{eq:integral} and~\eqref{eq:sqaureR}).

\State Evaluate (accurate values of) the Fourier coefficients of the characteristic function of $\Omega$ and the corresponding truncated Fourier expansion~\eqref{eq:fourierchar}.
\State Evaluate $I_2$ at all grid points $x_{i,j}$ (eq.~\eqref{eq:convsum}) by applying the FFT and inverse FFT with zero padding.

\For{$x_{i,j} \in \Omega$}
\State Initialize $I(x_{i,j}) = 0$.
    \If{$B_{w_1}(x_{i,j}) \cap \partial\Omega = \emptyset$}
        \State Compute $I_1(x_{i,j})$ using the quadrature rule~\eqref{eq:quadratureI_smooth}. 
        \State Update $I(x_{i,j}) \leftarrow I_1(x_{i,j})+I_2(x_{i,j})$
    \ElsIf{$B_{w_1}(x_{i,j}) \cap \partial\Omega \neq \emptyset$ \ \textbf{and} \ the intersection contains no corner points}
        \State Compute the distance function \eqref{eq:distancefunction}.
        \State Compute $I_1(x_{i,j})$ using the quadrature rule~\eqref{eq:quadratureI_smooth}.
        \State Update $I(x_{i,j}) \leftarrow I_1(x_{i,j})+I_2(x_{i,j})$
    \Else 
        \State Compute the distance functions \eqref{eq:fred} and \eqref{eq:fblue}.
        \State Compute $I_1$ using the corner-handling quadrature rule~\eqref{eq:quadratureI_corner}.
        \State Update $I(x_{i,j}) \leftarrow I_1(x_{i,j})+I_2(x_{i,j})$
    \EndIf
\EndFor

\end{algorithmic}
\end{algorithm}

\section{Numerical Results}
\label{sec:numerics}

This section illustrates the behavior of the proposed algorithm on two representative domains $\Omega$, namely, the unit disc
and the drop-shaped domain depicted in Figure~\ref{fig:teardrop}. 
For the first test case,
\[
  \Omega = \{\, y \in \mathbb{R}^2 : \|y\| \le 1 \,\},
\]
we take the periodicity domain
\[
  R = \big[-\tfrac{P}{2},\, \tfrac{P}{2}\big]^2,
\qquad P = 3,
\]
so that, in particular, $\Omega \subset R$.
In this setting, the required Fourier coefficients of $\chi_\Omega$
admit closed-form expressions.
Specifically,
\[
  (\chi_\Omega)_{00} = \frac{\pi}{P^2},
\]
and for $(m,n) \ne (0,0)$,
\begin{equation}\label{eq:F-coeffs}
  (\chi_\Omega)_{mn}
    = \frac{1}{P^{2}}
      \int_{\Omega}
      e^{-\mathrm{i} m \frac{2\pi}{P} y_1
         - \mathrm{i} n \frac{2\pi}{P} y_2}
      \, \mathrm{d}y_1 \mathrm{d}y_2
    = \frac{1}{P^{2}}
      \frac{\pi\,
            J_1\!\left(
              \frac{2\pi}{P} \sqrt{m^{2}+n^{2}}
            \right)}
           {\sqrt{m^{2}+n^{2}}}.
\end{equation}
Here $J_1$ denotes the Bessel function of the first kind. A closed-form expression is also available for the
convolution integral~\eqref{eq:integral} with $\phi(y)=1$ on this domain,
which provides a convenient reference solution for comparison, namely
\begin{equation}
\label{eq:exactcircleres}
    \int_{\Omega}\log\abs{x-y} \dif y = \frac{1}{2}\pi(\norm{x}^2-1)\quad x \in \Omega.\end{equation}
Using the equi-spaced grid 
\begin{equation}
    x_{ij} = (-\frac{P}{2} +i\frac{P}{2N},-\frac{P}{2} +j\frac{P}{2N}),
\end{equation}
on $R$, with $N = 2^k$ ($1\leq k\leq 10$). Algorithm~\ref{alg:alg_1} was applied in this context by employing
the explicit Fourier coefficients~\eqref{eq:F-coeffs} together with the FFT algorithm to produce the necessary values 
\begin{equation}
    \chi_{\Omega}^N(y_1,y_2)= \sum_{n,m = -N/2}^{N/2} (\chi_\Omega)_{mn} e^{\ii m y_1 + \ii n y_2}.
\end{equation}
of the truncated Fourier series $\chi_{\Omega}^N$ can be rapidly evaluated. We analyze first the choice of the numbers of discretization points in radial and angular directions, denoted by $N_r$ and $N_\theta$. These choices impact solely the computation of $I_1(x)$. By fixing $N=2^{11}$ we obtain machine precision for the error in $I_2(x)$. We compare the numerical results with the exact result obtained using \eqref{eq:exactcircleres} and obtain the following table of errors. 

\begin{table}[ht]
\centering
\caption{Error table for $I(x=(0.75,0.50))$ with $\phi = 1$. 
Here we choose the window widths $(w_1,w_0)=(1/2,1/6)$ and $N = 2^{11}$.}
\resizebox{\linewidth}{!}{
\begin{tabular}{c|cccccccc}
\hline
$N_\theta \backslash N_r$
 & $2^5$ & $2^6$ & $2^7$ & $2^8$ & $2^9$ & $2^{10}$ & $2^{11}$ & $2^{12}$ \\
\hline

$2^5$ &
$1.2\times10^{-4}$ & $4.3\times10^{-6}$ &
$6.9\times10^{-6}$ & $6.7\times10^{-6}$ & $6.7\times10^{-6}$ &
$6.7\times10^{-6}$ & $6.7\times10^{-6}$ & $6.7\times10^{-6}$ \\

$2^6$ &
$1.2\times10^{-4}$ & $1.1\times10^{-5}$ &
$2.2\times10^{-7}$ & $3.2\times10^{-8}$ & $3.3\times10^{-8}$ &
$3.3\times10^{-8}$ & $3.3\times10^{-8}$ & $3.3\times10^{-8}$ \\

$2^7$ &
$1.2\times10^{-4}$ & $1.1\times10^{-5}$ &
$1.9\times10^{-7}$ & $1.5\times10^{-9}$ & $2.0\times10^{-9}$ &
$2.0\times10^{-9}$ & $2.0\times10^{-9}$ & $2.0\times10^{-9}$ \\

$2^8$ &
$1.2\times10^{-4}$ & $1.1\times10^{-5}$ &
$1.9\times10^{-7}$ & $4.7\times10^{-10}$ & $1.2\times10^{-11}$ &
$1.1\times10^{-11}$ & $1.1\times10^{-11}$ & $1.1\times10^{-11}$ \\

$2^9$ &
$1.2\times10^{-4}$ & $1.1\times10^{-5}$ &
$1.9\times10^{-7}$ & $4.7\times10^{-10}$ & $1.2\times10^{-11}$ &
$1.1\times10^{-11}$ & $1.1\times10^{-11}$ & $1.1\times10^{-11}$ \\

$2^{10}$ &
$1.2\times10^{-4}$ & $1.1\times10^{-5}$ &
$1.9\times10^{-7}$ & $4.7\times10^{-10}$ & $1.2\times10^{-11}$ &
$1.1\times10^{-11}$ & $1.1\times10^{-11}$ & $1.1\times10^{-11}$ \\

$2^{11}$ &
$1.2\times10^{-4}$ & $1.1\times10^{-5}$ &
$1.9\times10^{-7}$ & $4.7\times10^{-10}$ & $1.2\times10^{-11}$ &
$2.3\times10^{-15}$ & $2.2\times10^{-15}$ & $2.1\times10^{-15}$ \\

$2^{12}$ &
$1.2\times10^{-4}$ & $1.1\times10^{-5}$ &
$1.9\times10^{-7}$ & $4.7\times10^{-10}$ & $1.2\times10^{-11}$ &
$2.1\times10^{-15}$ & $2.2\times10^{-15}$ & $2.1\times10^{-15}$ \\
\hline
\end{tabular}
}
\end{table}

By sweeping through the parameters, we display the following table denoting the minimal parameter triples needed to achieving a certain prescribed error.
\begin{table}[ht]
\centering
\caption{Minimal parameter triples $(N,N_r,N_\theta)$ achieving error $<\varepsilon$.}
\resizebox{\linewidth}{!}{
\begin{tabular}{c|cccccccccccc}
\hline
$\varepsilon$ 
& $10^{-3}$  & $10^{-5}$ 
 & $10^{-7}$ & $10^{-9}$  
 & $10^{-14}$ \\
\hline
$(N, N_r, N_\theta)$
& $(2^{6}, 2^{5}, 2^{4})$
& $(2^{8}, 2^{6}, 2^{5})$
& $(2^{9}, 2^{8}, 2^{6})$
& $(2^{10}, 2^{8}, 2^{8})$
& $(2^{11}, 2^{10}, 2^{10})$ \\
\hline
\end{tabular}
}
\end{table}
In this part of the section, we present numerical results for the drop shape domain. Its boundary is defined using the curve as presented in Figure \ref{fig:teardrop}.
\begin{equation}
    \{(3\sin(t/2)-\frac{3}{2},-2\sin(t))\ |\ t\in [0,2\pi]\}.
\end{equation}
For specific domains like this, there are no exact solutions. We compare the results with the refined version of our method ($2^{12}$ discretization points in both angular and radial direction for $I_1$ and $2^{12}\times 2^{12}$ mesh for $I_2$) as a comparison and report the convergence analysis and evaluation results in Figure \ref{fig:conv_dropshape}
\begin{figure}[ht]
\centering
\begin{subfigure}[b]{0.45\textwidth}
    \centering
    \includegraphics[width=\linewidth]{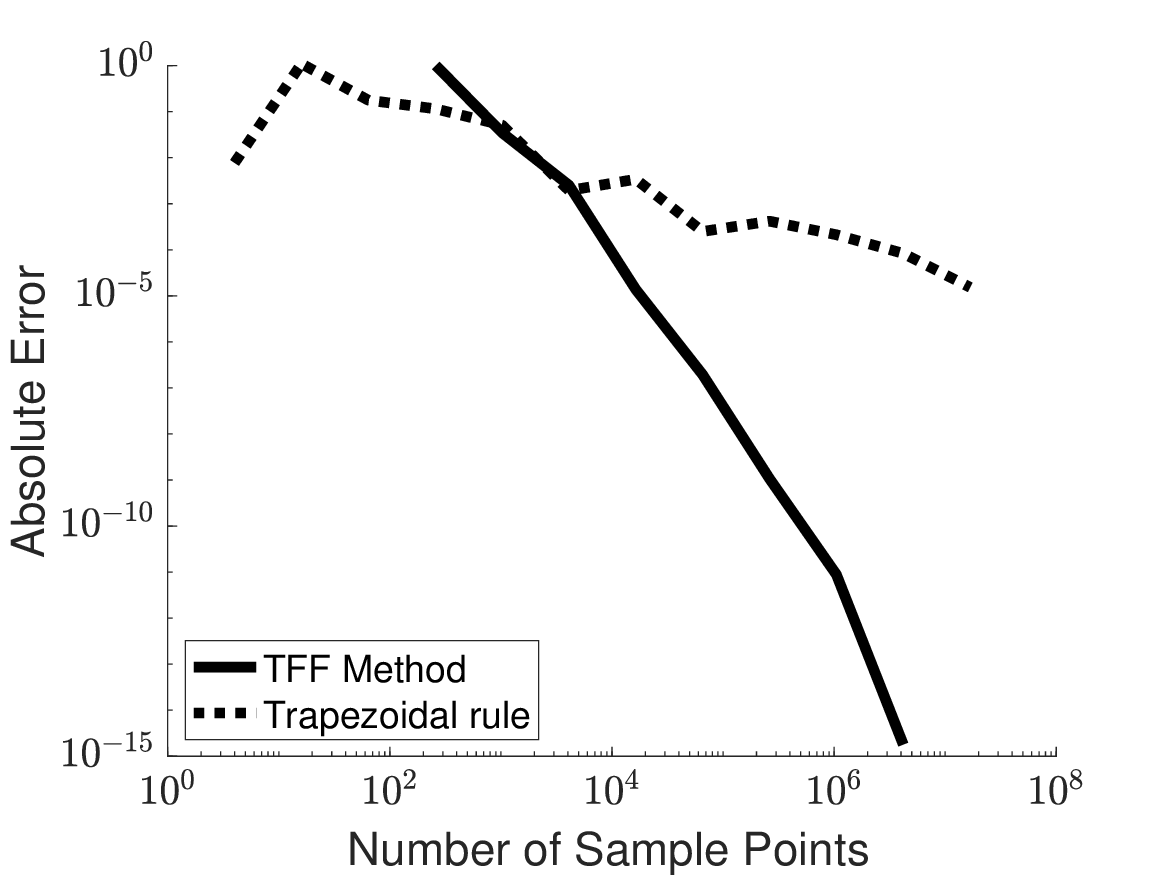}
     \caption{Convergence analysis for an $x$, where the window function lies entirely in the domain.}
\end{subfigure}
\begin{subfigure}[b]{0.45\textwidth}
    \centering
    \includegraphics[width=\linewidth]{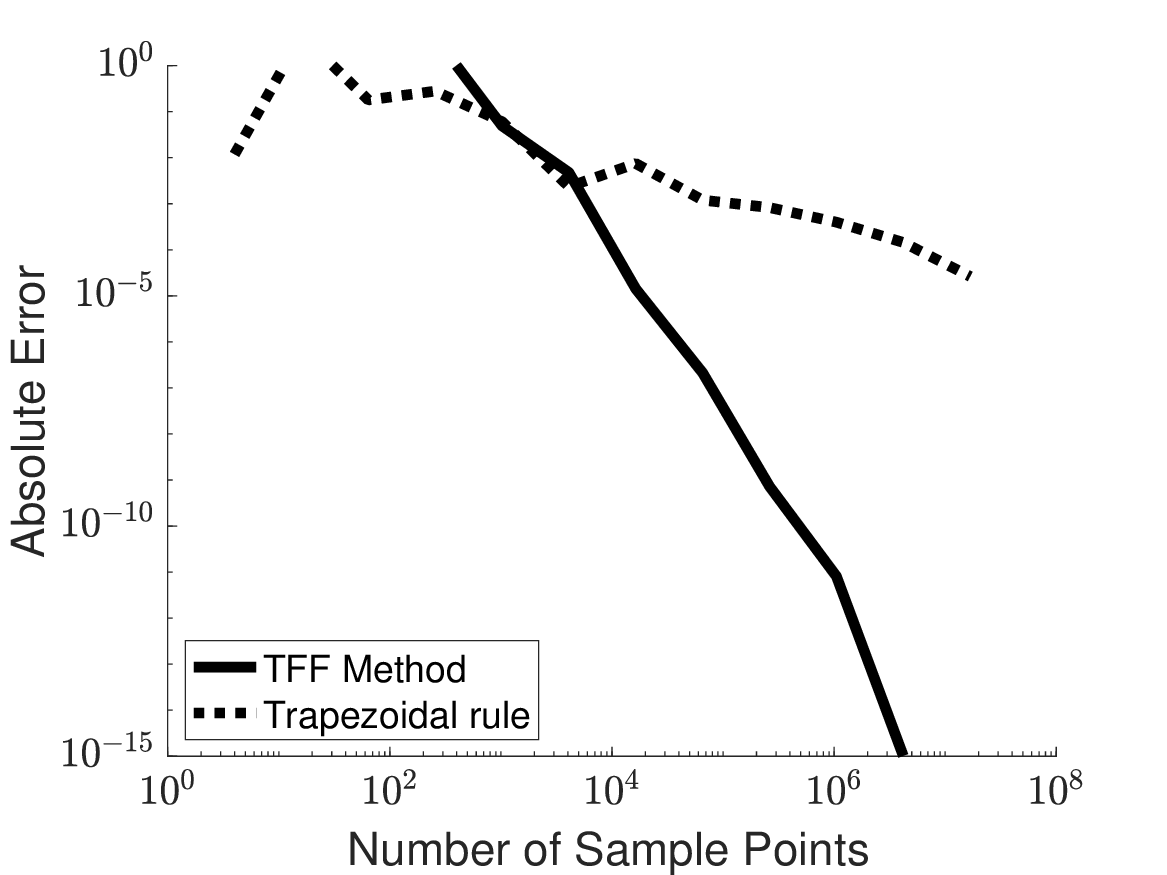}
    \caption{Convergence analysis for an $x$, where the window function intersections with the non-smooth corner of the domain.}
\end{subfigure}

    \caption{Convergence analysis for $I_2(x)$ with increasing number of equidistant sample points. Target value $2^{12}$ discretization point in each direction. Here we take $\phi(y) = \exp(\ii 40 y_1-\ii 20y_2)$.}
    \label{fig:part2convergence}
\end{figure}

\begin{figure}[ht]

    \centering
\begin{subfigure}[b]{0.45\textwidth}
    \centering
    \includegraphics[width=\linewidth]{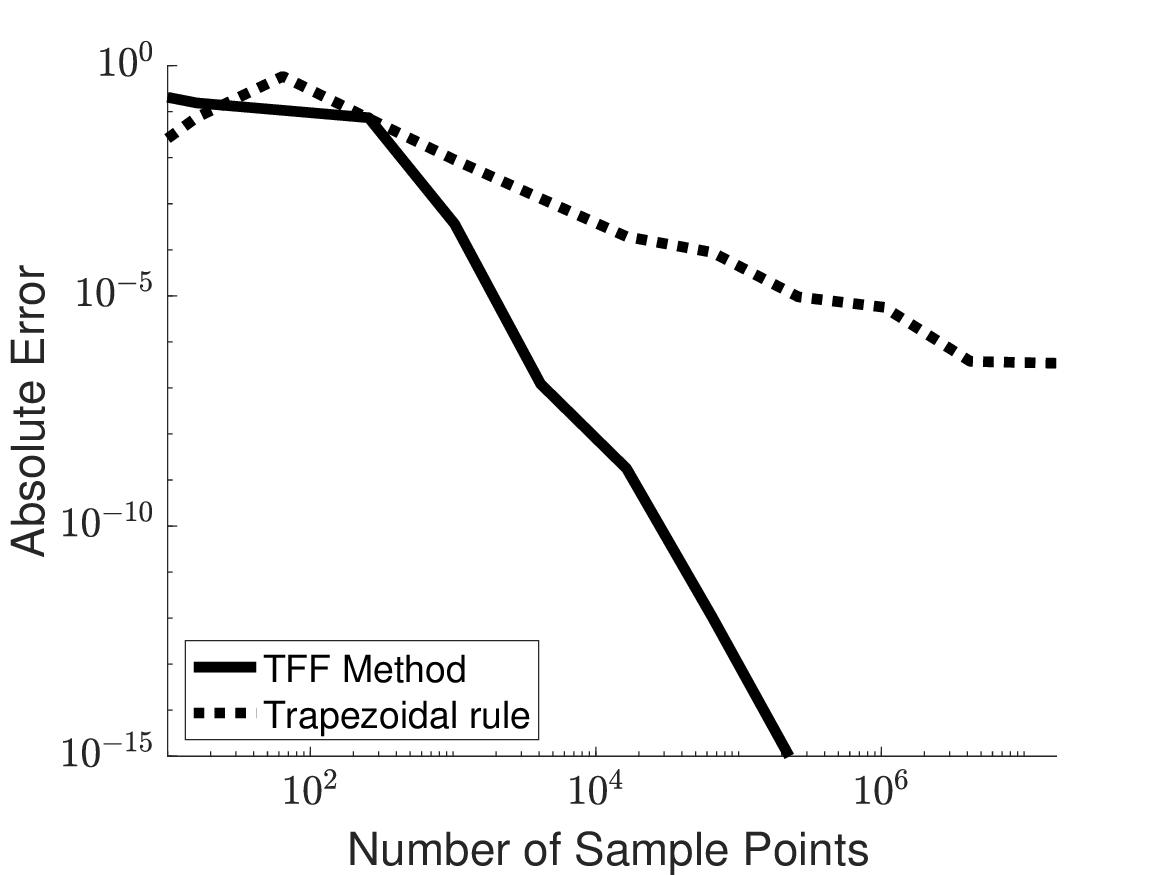}
    \caption{Convergence analysis for an $ x$ where the window function intersects smoothly.}\label{fig:6-4}
\end{subfigure}
\begin{subfigure}[b]{0.45\textwidth}
    \centering
    \includegraphics[width=\linewidth]{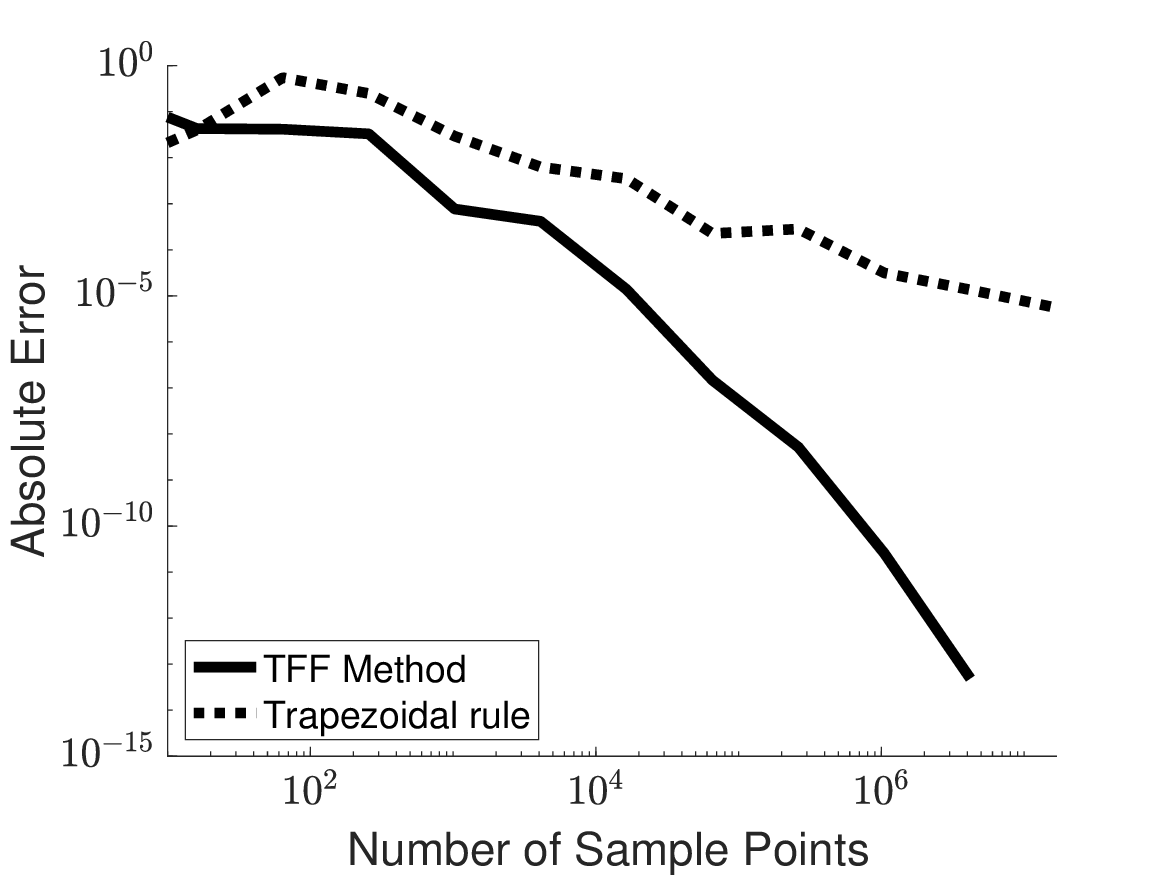}
    \caption{Convergence analysis for an $ x$ where the window function intersects with a corner.}\label{fig:6-6}
\end{subfigure}
    \caption{Convergence analysis for $I_1(x)$ with increasing $N_\theta = N_r$. Number of sample points is the product. Target value $2^{12}$ discretization point in each direction. Here we take $\phi(y) = \exp(\ii 40 y_1-\ii 20y_2)$.}
    \label{fig:threeintersections}
\end{figure}
\begin{figure}[ht]
    \centering
    \includegraphics[width=0.75\linewidth]{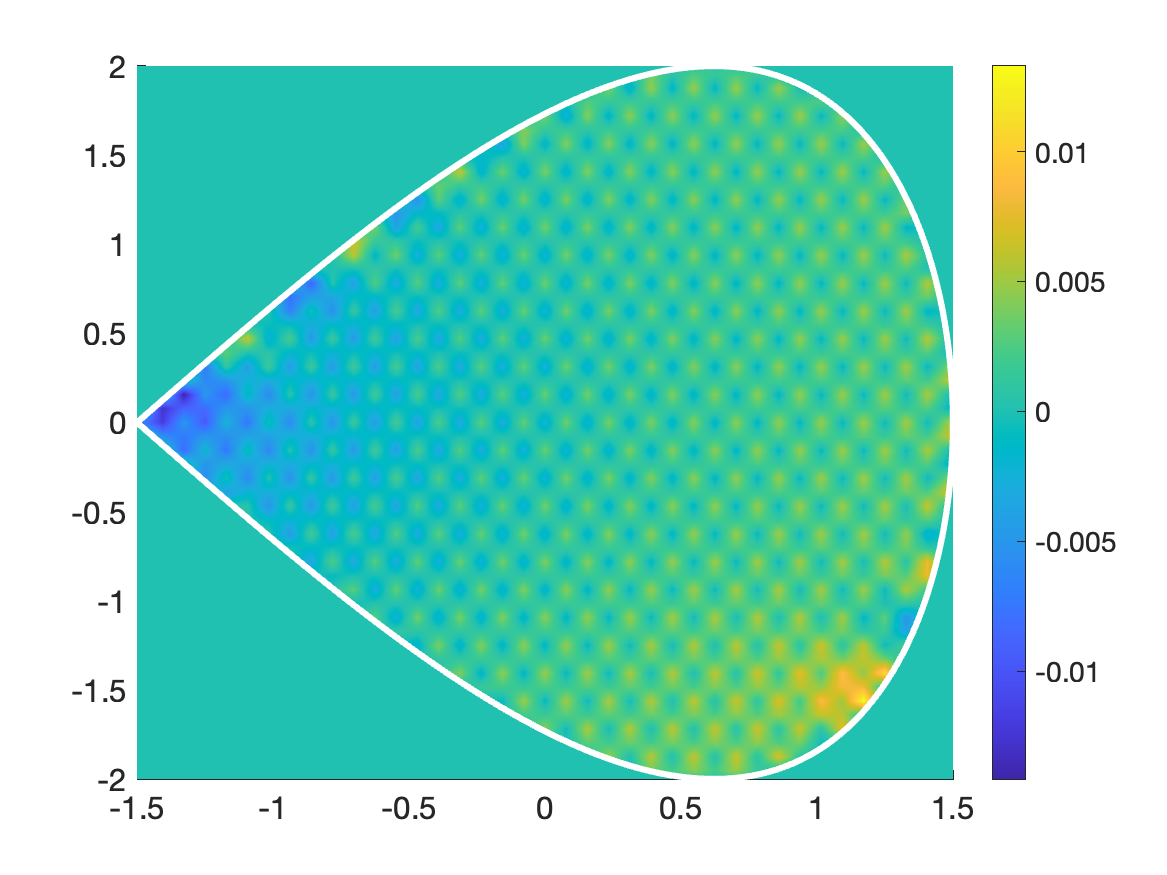}
    \caption{Numerical results for \eqref{eq:integral} in a drop-shaped domain with discretization parameters $N=2^9, N_\theta = 2^9$ and $N_r = 2^8$. Here we take $\phi(y) = \exp(\ii 40 y_1-\ii 20y_2)$. }
    \label{fig:conv_dropshape}
\end{figure}

In Figure \ref{fig:starring} we compute Equation \eqref{eq:integral} for the \emph{star-ring-shaped} domain considered in \cite{fryklund2023fmm}, whose inner and outer boundaries are given by the following perturbed polar parameterization
\begin{equation}
    \big(a_{\mathrm{o}}R_{\mathrm{o}}(\theta)\cos\alpha_{\mathrm{o}}(\theta),\ 
   a_{\mathrm{o}}R_{\mathrm{o}}(\theta)\sin\alpha_{\mathrm{o}}(\theta)\big)
\quad\text{and}\quad
\big(a_{\mathrm{i}}R_{\mathrm{i}}(\theta)\cos\alpha_{\mathrm{i}}(\theta),\ 
   a_{\mathrm{i}}R_{\mathrm{i}}(\theta)\sin\alpha_{\mathrm{i}}(\theta)\big),
\end{equation}
where 
\(
\theta\in[0,2\pi], 
R_{\mathrm{o/i}}(\theta)=2+\varepsilon_{R,\mathrm{o/i}}\sin(7\theta)
\)
and 
\(
\alpha_{\mathrm{o/i}}(\theta)=\theta+\varepsilon_{\Phi,\mathrm{o/i}}\sin(7\theta).
\)
\begin{figure}[ht]
    \centering
    \includegraphics[width=0.75\linewidth]{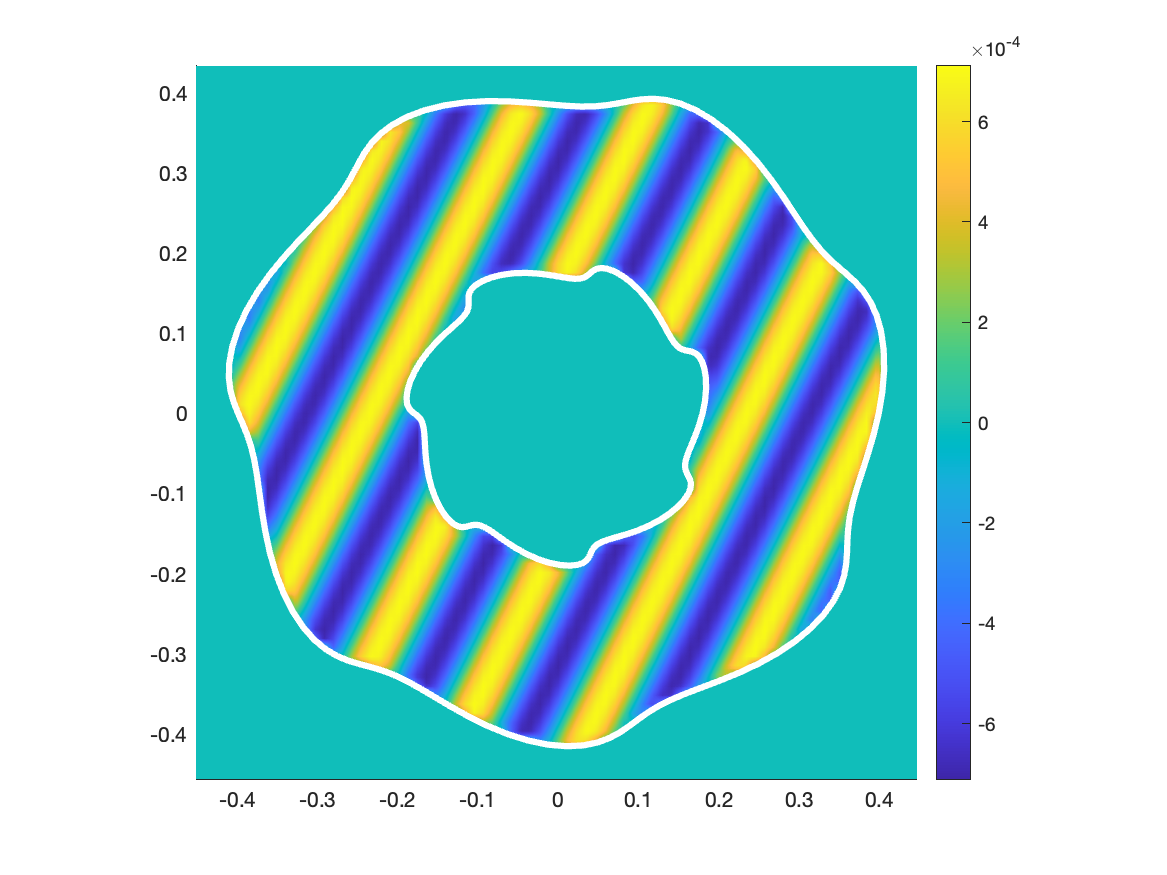}
    \caption{$N=2^8, N_\theta = 2^8, N_r = 2^8$ Error $10^{-10}$ and Runtime $4$ seconds for $\approx 4000$ points inside the domain. Here we take $\phi(y) = \exp(\ii 40 y_1-\ii 20y_2)$. }
    \label{fig:starring}
\end{figure}
\begin{remark}
The timings reported in this section correspond solely to the TFF convolution stage, since the Laplace integral‐equation solver is not accelerated. If an accelerated Laplace solver were used, its computational cost would be negligible compared with that of computing the TFF-based particular convolution solution. Additionally, these results reflect the present Matlab implementation, in spite of which the algorithm results in run times ranging from a fraction of a second to four seconds.
\end{remark}
\newpage
\appendix
\section{Some results on Fourier series}
\subsection{Decay of Fourier coefficients}
\label{sec:decay}
The Fourier series of a periodic function $f$ on an interval $[0,P]$ is defined as
\begin{equation}
    f(x) = \sum_{m\in\mathbb{Z}} f_m e^{\ii \frac{2\pi}{P}x},\quad f_m = \frac{1}{P}\int_{0}^P f(x)e^{-\ii \frac{2\pi}{P}x}\ \dif x.
\end{equation}
Using integration by parts, we can prove the following lemma.
\begin{lemma}
\label{lem1}
    If $f\in C^k$, then we have $\abs{f_m}\leq \frac{\norm{f^{(k)}}}{\abs{m}^k}$.
\end{lemma}
In higher dimensions, we can use multi-index to express the Fourier coefficients of a periodic function. A generalization of the above lemma can be found in \cite{plonka}.
\begin{lemma} (\cite[Lemma 4.6]{plonka})
\label{lem2}
    Let $r \in \mathbb{N}$ be given. If $f$ and its partial derivatives $D^\alpha f$ are contained in $L_1\left(\mathbb{T}^d\right)$ for all multi-indices $\boldsymbol{\alpha} \in \mathbb{N}_0^d$ with $|\boldsymbol{\alpha}| \leq r$, then it holds for the Fourier coefficients $c_{\mathbf{k}}(f)$:
\begin{equation}
\lim _{\|\mathbf{k}\|_2 \rightarrow \infty}\left(1+\|\mathbf{k}\|_2^r\right) c_{\mathbf{k}}(f)=0    
\end{equation}
\end{lemma}
\subsection{Computation of the Fourier series coefficients for \texorpdfstring{$x\log(x)$}{x log(x)}}
\label{sec:Cnrlogr}
This appendix presents an efficient numerical method for the accurate evaluation of the Fourier coefficients 
\begin{equation}\label{eq:Ln}
        L_n  = \frac{1}{P} \int_{0}^{P} \chi_{[0,r]}\log(t) t \exp(-\mathrm{i}2\pi\frac{n}{P}t)\  \mathrm{d}t=  \frac{1}{P} \int_0^{r} \log(t) t \exp(-\mathrm{i}2\pi\frac{n}{P}t)\  \mathrm{d}t.
\end{equation}
in the Fourier expansion
\begin{equation}
        L(x) = \sum_{n= -\infty}^\infty L_n \exp(\mathrm{i}2\pi\frac{n}{P}x),
\end{equation}
of the function $L(x) = x\log(x)\chi_{[0,r]}(x)$ in the periodicity interval $[0,P]$, where $r\in [0,P]$ is a given real number. The right-hand integral in~\eqref{eq:Ln} can be  computed explicitly in terms of the
the exponential integral
\begin{equation}
        \mathrm{Ei}(x) = \int_{-\infty}^x \frac{e^t}{t} \ \dif t,
\end{equation}
indeed, we have
\begin{equation}
        L_n = \frac{1}{P}(\frac{e^{-\mathrm{i} 2\pi\frac{n}{P} r}(\mathrm{i} 2\pi\frac{n}{P} r \log (r)+\log (r)+1)-\operatorname{Ei}(-\mathrm{i} 2\pi\frac{n}{P} r)}{(2\pi\frac{n}{P})^2} -\frac{1-\gamma-\log(-\mathrm{i}2\pi\frac{n}{P})}{(2\pi\frac{n}{P})^2})
\end{equation}
    for $n\neq 0$, and
\begin{equation}
        L_0 = \frac{1}{4P} r^2 (2\log(r)-1).
\end{equation}

Unfortunately, the evaluation of Ei$(r)$ at the large set of $r$ values required per Section~\ref{sec:windowpart1} requires an unacceptable computing cost. The proposed efficient algorithm for the evaluation of the coefficients $L_n(x)$ employs Chebyshev expansions, as described in what follows.    

To derive the desired algorithm we first use integration by parts to re-express the right-hand integral in~\eqref{eq:Ln} in the form
    \begin{equation}
    \label{eq:Cnpi}
   \frac{P}{4\pi^2 n^2}
\left[
\log(r)\left( e^{-\ii\frac{2\pi n}{P} r}\Bigl(1+\ii\frac{2\pi n}{P}r\Bigr) -1 \right)
-\int_0^r \frac{e^{-\ii\frac{2\pi n}{P} t}(1+\ii\frac{2\pi n}{P} t)-1}{t}\,dt
\right], 
\end{equation}
which reduces the evaluation of \eqref{eq:Cnpi} to evaluating the integral
\begin{equation}
        \ell(n) = \int_0^r \frac{e^{-\ii \frac{2\pi n}{P} t} (1+\ii \frac{2\pi n}{P} t)-1}{t} \dif t
\end{equation}
or, using the substitution $s = nt$, the integral
    \begin{equation}
    \label{eq:Cn_cheb}
         \ell(n) = \int_0^{nr} \frac{e^{-\ii  \frac{2\pi}{P} s}(1+\ii \frac{2\pi}{P} s)-1}{s} \dif s,
    \end{equation}
    for $n=1,\dots, F$, This approach thus
     reduces the calculation of the coefficients $L_n$ to the evaluation of the primitive of the integrand in~\eqref{eq:Cn_cheb} at the points $nr$ with $n=1,\dots, F$, all of which lie in the interval $0\leq nr\leq w_1F$.
     
    To evaluate this primitive we thus partition the interval $0\leq s\leq Fw_1$ into a number $m$ of subintervals 
    \begin{equation}
        B_i = \left[(i-1)H,i H\right]\quad 1\leq i\leq m
    \end{equation}
    of length $H = w_1F/m$. Then, for a given point  $q_n=\floor{nr/H}$,  the integral \eqref{eq:Cn_cheb} that we intend to compute becomes
    \begin{equation}
         \ell(n)  =\sum_{i=1}^{q_n} \int_{B_i}   
         \frac{e^{-\ii \frac{2\pi}{P} s}(1+\ii \frac{2\pi}{P} s)-1}{s}  \dif s + \int_{q_n H}^{nr} \frac{e^{-\ii \frac{2\pi}{P} s}(1+\ii \frac{2\pi}{P} s)-1}{s}  \dif s.
         \end{equation}
    Then, defining the change of variables $c_i= \frac{2}{H}(x- H(i-1))-1$ and letting $$h_i(x) =\frac{H}{2}\frac{e^{-\ii \frac{2\pi}{P} c_i(x)}(1+\ii \frac{2\pi}{P} c_i(x))-1}{ c_i(x)}$$ 
    we obtain
    \begin{equation}
    \label{eq:chebycnfinal}
        \ell (n) = \sum_{i=1}^{q_n} \int_{-1}^1 h_i(x)  \dif x + \int_{-1}^{c_{q_n+1}(nr)} h_{q_n+1}(x)  \dif x. \\
    \end{equation}
   Each one of the integrals in~\eqref{eq:chebycnfinal} may be computed effectively by employing the Chebyshev expansions~\cite{trefethen2000spectral}.
    \begin{equation}
    \label{eq:chebsum}
        h_i(x) = \sum_{m\in \mathbb{N}} a_m^i T_m(x),
    \end{equation}
    in terms of the  Chebyshev polynomials of the first kind 
    \begin{equation}\label{Tm}
     T_m(x) = \cos(m\arccos(x)).    
    \end{equation}
    Indeed, inserting  \eqref{eq:chebsum} into \eqref{eq:chebycnfinal} we obtain the expressions
    \begin{equation}
    \label{eq:chebycnfinalsum}
        \ell(n) = \frac{H}{2} \left( \sum_{i=1}^{q_n} \sum_{m\in\mathbb{N}} a_m^i \int_{-1}^1 T_m(x)\dif x +  \sum_{m\in\mathbb{N}} a_m^{q_n+1} \int_{-1}^{c_{q_n+1}(nr)} T_m(x)\dif x\right),
    \end{equation} 
    which can be evaluated by employing the relations
    \begin{equation}
         \int_{-1}^1 T_m(x)\ \dif x= \begin{cases}\frac{(-1)^m+1}{1-m^2} & \text { if } m \neq 1 \\ 0 & \text { if } m=1\end{cases}
    \end{equation}
    and
    \begin{equation}
        \int T_m(x) d x = \frac{1}{2}\left(\frac{T_{m+1}(x)}{m+1}-\frac{T_{m-1}(x)}{m-1}\right)+\text {constant,\ for }m\geq 1.
    \end{equation}
    The necessary expansion coefficients $a_m^i$ and values of Chebyshev polynomials 
    at specific points that are needed to evaluate \eqref{eq:chebycnfinalsum} can be obtained by employing standard numerical techniques for spectral numerical methods. In particular, in view of the relation~\eqref{Tm} the Chebyshev coefficients $a_m^i$ coincide with the coefficients of a cosine expansion, which can be obtained efficiently by employing the Fast Cosine Transform (FCT). The necessary point values, in turn, can be obtained, without recourse to expensive evaluation of the trigonometric functions,  by employing the three-term recursion relation
    \begin{equation}
        T_0(x) = 1,\ T_1(x) = x,\ T_{m+1}(x) = 2x T_m(x) - T_{m-1}(x).
    \end{equation}
\subsection{Fourier series of the characteristic function of a parametrized domain}
\label{sec:divergence}

Fourier series of indicator functions are known in closed form for several simple domains. 
For example, consider the disk 
\[
D_R=\{(y_1,y_2)\in\mathbb{R}^2 : y_1^2+y_2^2\le R^2\}
\]
of radius \(R>0\).  
The Fourier coefficients of its characteristic function on 
\([{-}P/2,P/2]\times[{-}P/2,P/2]\) satisfy  
\begin{equation}
(\chi_{D_R})_{m,n}=
\begin{cases}
\dfrac{\pi R^{2}}{P^{2}}, & m=n=0,\\[6pt]
\dfrac{2\pi\, J_1(\rho_{m,n})}{\rho_{m,n} P^{2}}, 
\quad \rho_{m,n}=\sqrt{\left(\dfrac{2\pi m}{P}\right)^{2}+\left(\dfrac{2\pi n}{P}\right)^{2}},
& \text{otherwise},
\end{cases}
\end{equation}
where \(J_1\) denotes the Bessel function of the first kind.

For a general parametrized domain $\Omega\subset S\subset\mathbb{R}^2$, 
the Fourier coefficients of its indicator function are defined by
\begin{equation}
(\chi_{\Omega})_{m,n}
= \int_{S} \chi_{\Omega}(y)\, e^{\mathrm{i}(m y_1 + n y_2)}\, \mathrm{d}A
= \int_{\Omega} e^{\mathrm{i}(m y_1 + n y_2)}\, \mathrm{d}A.
\end{equation}

If $m=n=0$, we have
\begin{equation}
(\chi_{\Omega})_{0,0} = \int_{\Omega} 1\,\mathrm{d}A = |\Omega|.
\end{equation}

In the case where $m\neq 0$, we apply integration by parts using
\[
\nabla\!\cdot
\begin{pmatrix}
e^{\mathrm{i}(m y_1 + n y_2)}\\[2pt]
0
\end{pmatrix}
= \mathrm{i} m\, e^{\mathrm{i}(m y_1 + n y_2)},
\]
which yields
\begin{equation}
\begin{aligned}
(\chi_{\Omega})_{m,n}
&= \frac{1}{\mathrm{i}m}
   \int_{\Omega}
   \nabla\!\cdot
   \begin{pmatrix}
   e^{\mathrm{i}(m y_1 + n y_2)}\\[2pt]
   0
   \end{pmatrix}
   \mathrm{d}A \\
&= \frac{1}{\mathrm{i}m}
   \int_{\partial\Omega}
   e^{\mathrm{i}(m y_1 + n y_2)}\, \mathrm{d}y_1,
\qquad m\neq 0.
\end{aligned}
\end{equation}

If $\partial\Omega$ admits a smooth parameterization  
$\gamma(t)=(\gamma_x(t),\gamma_y(t))$, $t\in[0,2\pi]$, the expression becomes
\begin{equation}
(\chi_{\Omega})_{m,n}
= \frac{1}{\mathrm{i}m}
  \int_{0}^{2\pi}
  e^{\mathrm{i}\big(m\gamma_x(t)+n\gamma_y(t)\big)}
  \,\gamma_x'(t)\,\mathrm{d}t,
\qquad m\neq 0.
\end{equation}

For $m=0$ but $n\neq 0$, an analogous formula is obtained by integrating 
with respect to $y_2$ instead, giving
\begin{equation}
(\chi_{\Omega})_{0,n}
= \frac{1}{\mathrm{i}n}
  \int_{\partial\Omega}
  e^{\mathrm{i}n y_2}\, \mathrm{d}y_2
= \frac{1}{\mathrm{i}n}
  \int_{0}^{2\pi}
  e^{\mathrm{i}n\gamma_y(t)}\, \gamma_y'(t)\,\mathrm{d}t,
\qquad n\neq 0.
\end{equation}

\printbibliography
\end{document}